\documentclass[11pt]{article}
\usepackage[utf8]{inputenc}
\usepackage[english]{babel}
\usepackage[utf8]{inputenc}
\usepackage{amsmath}
\usepackage{subfigure}
\usepackage{paralist}
\usepackage[numbers]{natbib}
\usepackage{verbatim}
\usepackage{graphicx}
\usepackage[colorinlistoftodos]{todonotes}
\usepackage{hyperref}
\usepackage{mathtools}

\usepackage{algorithm, algorithmic} 
\usepackage{multicol}
\usepackage[margin=2.54cm]{geometry}
\usepackage{setspace}
\usepackage{amsmath,amsthm,amsfonts,amscd,amssymb} 

\usepackage{verbatim}      	
\usepackage{makeidx}       	
\usepackage{url}		
\usepackage{graphicx}
\usepackage{epstopdf}
\usepackage[subpreambles=true]{standalone}
\usepackage{bm}
\usepackage{graphicx}
\usepackage{indentfirst}
\usepackage{color}
\usepackage{enumerate}

\usepackage{tikz}
\usepackage{epstopdf}
\usepackage{paralist}
\usepackage{xspace}
\usepackage[group-separator={,}]{siunitx} 
\usetikzlibrary{arrows}
\usepackage[super]{nth}
\usepackage{wrapfig}
\usepackage{framed}
\usepackage{subfiles}

\usepackage{tabularx,hhline}
\usepackage{multirow}
\usepackage{booktabs}
\usepackage{cancel}

\usepackage{algorithm}
\usepackage{algorithmic}
\usepackage{amsthm}

\newcommand{\st}{\textnormal{s.t.}}

\DeclareMathOperator*{\argmax}{arg\,max}

\newcommand{\RR}{\mathbb{R}}

\usepackage{tcolorbox}

\newcommand{\rhot}{\tilde\rho}

\newtheorem{thm}{Theorem}[section]

\newtheorem{lem}[thm]{Lemma}
\newtheorem{prop}[thm]{Proposition}

\newtheorem{defn}{Definition}

\newtheorem{rem}{Remark}[section]

\begin{document}
\title{Distributionally Robust Observable Strategic Queues}

\author{Yijie Wang,   Madhushini Narayana Prasad,  Grani A. Hanasusanto, and John J. Hasenbein\\ Graduate Program in Operations Research and Industrial Engineering\\ The University of Texas at Austin, USA}

\date{\today}

\maketitle

\begin{abstract}

\noindent This paper presents an extension of Naor's analysis on the join-or-balk problem in observable M/M/1 queues. While all other Markovian assumptions still hold, we explore this problem assuming uncertain arrival rates under the distributionally robust settings. We first study the problem with the classical moment ambiguity set, where the support, mean, and mean-absolute deviation of the underlying distribution are known. Next, we extend the model to the data-driven setting, where decision makers only have access to a finite set of samples. We develop three optimal joining threshold strategies from the perspective of an individual customer, a social optimizer, and a revenue maximizer, such that their respective worst-case expected benefit rates are maximized. Finally, we compare our findings with Naor's original results and the traditional sample average approximation scheme. 

\mbox{}


\end{abstract}



\section{Introduction} \label{sec:intro}
	
Imposing tolls to regulate queueing systems was first studied by Naor \cite{naor1969regulation}. 
He considers a single-server first-come-first-served (FCFS) queue with stationary Poisson arrivals at a known rate $\lambda$. Service times are independent, identically, and exponentially distributed with the rate $\mu$. Customers are assumed to be risk-neutral and homogenous from an economic perspective. Each customer receives a reward of $\$R$ upon service completion and incurs a cost of $\$C$ per unit of time spent in the system (including in service). In the observable model, every arriving customer inspects the queue length and decides whether to join (reneging is not allowed) or balk (i.e., not join the queue). This strategic decision making is the key factor differentiating this model from the classic $M/M/1$ queueing model. 
	
	Naor derives an optimal threshold strategy $n$: the customer joins the queue if and only if the system length is less than ${n}$. He computes this threshold value under three different control strategies: 1)  \emph{individual optimization} ($ n_e$) where the customers act in isolation aiming to maximize their own expected net benefit rate, 2) \emph{social optimization} ($ n_s$) where the objective is to maximize the long-run rate at which customers accrue net benefit and, 3) \emph{revenue maximization} ($n_r$) where the agency imposes a toll on the customers joining the queue with the goal of maximizing its own revenue. The most important result by Naor is the relation $ n_r \leq  n_s \leq  n_e$, which implies that the customers tend to join the system at a higher rate, when left to themselves, than is socially optimal. This is because customers do not consider the negative externalities they impose on customers who arrive later. The result also implies that the revenue maximizing firms allow fewer customers to join their system than the socially optimal case. 
	
	Many authors have expanded on the seminal work by Naor \cite{naor1969regulation}---a detailed review of these game-theoretic models is presented in a recent book by Hassin and Haviv \cite{hassin2003queue}. Some of the other recent works  \cite{burnetas2007equilibrium, economou2008equilibrium, guo2011strategic} involve deriving threshold strategies in a classic Naor's setting with server shutdowns. While Economou and Kanta \cite{economou2008equilibrium} study the system with server breakdowns and repairs, Burnetas and Economou \cite{burnetas2007equilibrium} analyze the system where the server shuts off when idle and incurs a set-up time to resume. A slight variant of this model is given by Guo and Hassin \cite{guo2011strategic} where the server resumes only when the queue length exceeds a given critical length. Also, Guo and Zipkin~\cite{guo2007analysis} explore the effects of three different levels of delay information and identify the specific cases which do and do not require such information to improve the performance. Haviv and Oz~\cite{haviv2016regulating} review the properties of several existing regulation schemes and devise a new mechanism where customers are given priority based on the queue length. 
	Afe\`che and Ata~\cite{afeche2013bayesian} study the observable $M/M/1$ queue with heterogenous customers, some patient and some impatient of given proportion. 
	
	All the aforementioned works explore the Naor's model by assuming deterministic arrival and service rates. Some recent studies have relaxed this restrictive assumption by taking the arrival or the service rate as a random variable. Debo and Veeraraghavan \cite{debo2014equilibrium} consider a system where the arriving customers cannot completely observe the service rate and value. They assume that the server belongs to one of two known types, and that the service rate and prior probability for each type is known. Liu and Hasenbein~\cite{CCLiu} study a stochastic extension of Naor's model by relaxing the assumption of certain arrival rate. They assume the arrival rate is drawn from a probability distribution  that is known to the decision maker. 
	Chen and Hasenbein~\cite{chen2020knowledge} further extend the stochastic model to the unobservable setting. They show that the social optimizer induces a lower expected arrival rate than the revenue maximizer in this setting. Hassin et al.~\cite{hassin2021strategic} also investigate the unobservable stochastic model from the perspective of strategic customers and demonstrate that the model exhibits a RASTA (rate-biased arrivals see time averages) property.
	However, these works still assume the distribution of the arrival or service rate is known precisely to decision makers, which may not be realistic in practice. In this paper, we extend the classical Naor's model for observable systems by relaxing these assumptions, where we assume the arrival rate is uncertain and governed by an unknown underlying distribution, while the service rate is deterministic.

    We consider an alternate modeling paradigm called the \emph{distributionally robust optimization} (DRO)~\cite{scarf1957min, shapiro2002minimax, vzavckova1966minimax}. Unlike the traditional stochastic optimization model, DRO acknowledges the lack of full distributional information on the random arrival rate. Instead, the decision maker is assumed to have access
    to partial information such as the moments and structural properties of the distribution, or some limited historical observations. In this setting, the objective is to derive optimal threshold strategies that maximize the worst-case expected benefit rate, where the worst case is taken over an \emph{ambiguity set} of all distributions  consistent with the available information about the true distribution. Such max-min problems have been studied since the seminal work by Scarf~\cite{scarf1957min} but only have received more attention with the advent of modern robust optimization techniques \cite{ben2009robust, bertsimas2004price}. 
	Since then, a substantial body of literature is devoted to studying well-known optimization problems under uncertainty in a distributionally robust setting; see \cite{ardestani2016linearized, delage2010distributionally,hanasusanto2015distributionally, li2014distributionally, shafieezadeh2015distributionally, wiesemann2014distributionally}. However, to the best of our knowledge, the distributionally robust framework has not been considered in the context of classical Naor's observable strategic queue model. The paper fills this gap in the literature. 
	

	We first study the distributionally robust queue model with a mean-absolute deviation (MAD) ambiguity set, where partial information about the distribution mean and MAD are known. 
    The use of the MAD ambiguity set is motivated by a recent work by Eekelen et al. \cite{van2022mad} who analyze the worst-case performance of the GI/G/1 queue under mean-dispersion constraints for the interarrival and service time distributions. The authors  demonstrate that measuring the dispersion by MAD, instead of variance, significantly simplifies the analysis and enables a closed-form solution for the extremal distribution whenever the loss function is convex. Inspired by this idea, we prove the concavity of the revenue rate function in the revenue maximization problem, which leads to an analytical solution for the worst-case expectation problem. Unfortunately, the social benefit rate function in the social optimization problem is neither concave nor convex. For this complicated function, we establish that, under some mild prerequisites, the function is unimodal and the MAD ambiguity set still admits a closed-form representation for the extremal distribution. When the prerequisites do not hold, we derive tractable reformulations for the social optimization problem.

	Next, we extend our model to the data-driven setting, where queue system managers only have access to a finite number of independent and identically distributed training samples collected from historical observations. We then construct a data-driven MAD ambiguity set which mitigates estimation errors from the empirical moment estimators. The distributionally robust model with a data-driven ambiguity set admits a semidefinite programming reformulation for the social optimization problem and a linear programming reformulation for the revenue maximization problem. 
	To properly determine the robustness parameters, we establish a new distribution-free confidence interval for the empirical MAD. Although such confidence intervals exist for the empirical mean and variance \cite{delage2010distributionally}, to the best of our knowledge, none is available for the empirical MAD: Herrey~\cite{herrey1965confidence} derives confidence interval for the empirical MAD under a normally distribution data, while other works mostly focus on median-absolute deviation; see \cite{abu2018confidence,arachchige2019confidence,bonett2003confidence}. Using this result, we further derive finite-sample guarantees of the data-driven MAD model, whose optimal value provides high confidence lower bounds on the expected social benefit or revenue rate. We also benchmark our data-driven MAD ambiguity set with the popular Wasserstein ambiguity set~\cite{esfahani2017data,esfahani2018data,gao2016distributionally,pflug2007ambiguity}, which is widely used in the data-driven setting as it can offer attractive finite-sample guarantees. Our results demonstrate that the data-driven MAD model shares a similar guarantee as the Wasserstein model while yields a much more efficient reformulation.

	Our main contributions of this paper can be summarized as follows. 
	\begin{enumerate}
        \item We propose a new  model to tackle the uncertain arrival rate in Naor's strategic queue problem using the emerging DRO framework. The model does not impose any specific distributional assumption; instead, it optimizes in view of the worst-case distribution within a prescribed ambiguity set. Benefitting from this robustification framework, the model alleviates the overfitting issue and yields attractive out-of-sample performance.

        \item We prove the revenue rate function is concave, while the social benefit rate function is either concave or unimodal under some mild prerequisites. We then show that these properties enable a closed-form solution for the worst-case expectation problem with a MAD ambiguity set. For the general cases, we derive a semidefinite programming (SDP) reformulation for the social optimization problem and a linear programming reformulation for the revenue optimization problem.
        
        \item We extend the distributionally robust model to the data-driven setting, where queue system managers only have access to a finite set of historical observations. To mitigate the adverse effect of the estimation errors from the empirical MAD, we robustify the MAD ambiguity set by adding an extra layer of robustness to the empirical mean and MAD estimators. The data-driven MAD model admits a SDP reformulation for the social optimization problem and a linear programming reformulation for the revenue maximization problem. We then establish a distribution-free confidence interval for the empirical MAD and derive finite-sample guarantees for the distributionally robust model with a data-driven MAD ambiguity set. Compared with the Wasserstein ambiguity set, the data-driven MAD ambiguity set admits a more efficient reformulation of fixed complexity,  where the number of constraints does not scale with the sample size.


	\end{enumerate}
	
	The remainder of the paper is structured as follows. In Section \ref{sec:droqueue}, we propose the distributionally robust queue model and analyze the relationship between different thresholds under the distributionally robust setting. Section \ref{sec:DROM} presents tractable reformulations for the worst-case expectation problem with a classical MAD ambiguity set. Section \ref{sec:data-driven-mad} explores the distributionally robust model with a data-driven MAD ambiguity set and derives theoretical finite-sample guarantees.
	Finally, the out-of-sample performances of our distributionally robust models are assessed empirically in Section~\ref{sec:numerical}.

	\paragraph{Notation:} The set of all probability measures supported on $\Xi$ is written as $\mathcal P_0(\Xi)\coloneqq\{\mu \in \mathcal M_+: \int_{\Xi} \mu(d\xi)=1 \}$,  where $\mathcal M_+$ denotes the set of nonnegative Borel measures. All random variables are designated by tilde signs (e.g., $\tilde\rho$), while their realizations are denoted without tildes (e.g., $\rho$). We denote by $\mathbb E_\mathbb P[c(\tilde\rho)]$ the expectation of a cost function with respect to random variable $\tilde\rho$ under distribution $\mathbb P$. We define $\lfloor n \rfloor$ to be  the largest integer less than or equal to $n$ and  $\|\bm x\|_p$ to be  the $p$-norm of a vector $\bm x$. For any set $\Xi$, we let $\text{int}(\Xi)$ denote its interior. 
	The cone of $k \times k$ positive semidefinite matrices is denoted by $\mathbb S_+^k$.
	
	\section{Distributionally Robust Strategic Queues Model}\label{sec:droqueue}
	The extension of Naor's seminal queue model to the stochastic optimization setting with an uncertain arrival rate was first proposed by Chen and Hasenbein~\cite{CCLiu} who  consider an $M/M/1$ queue system with a random arrival rate $\tilde \lambda\sim \mathbb P^\star$ and a deterministic service rate $\mu$. The queue system operates under a first-come-first-served discipline, and the true distribution of the uncertain arrival rate $\tilde \lambda$ is known by the system manager. Since the service rate $\mu$ is deterministic, without loss of generality, we consider the traffic intensity~$\rhot :=\frac{\tilde\lambda}{\mu} $ as the uncertain parameter throughout the remainder of the paper. The stochastic model aims to find an optimal threshold that maximizes the expected benefit rate, i.e.,
	\[\max_{n \in \mathbb Z_+} \mathbb E_{\mathbb P^\star}[c_n(\tilde \rho)].
	\]
	Here $c_n(\tilde \rho)$ is a general return function, which can be replaced with the social benefit rate function or revenue rate function, depending on the system manager's objective. 
	
	In practice, the true distribution $\mathbb P^\star$ is never available to the system manager and typically has to be estimated using the empirical distribution generated from the historical observations. While the empirical-based methods may work well on the observed data set, they often fail to achieve an acceptable out-of-sample performance because they do not consider any possible disturbances from the limited historical observations. 
	
	In this paper, we endeavor to address this fundamental shortcoming using ideas of DRO. The DRO approach does not impose any single distribution on the uncertain arrival rate. Instead, it constructs an ambiguity set $\mathcal P$ containing all plausible probability distributions that are consistent with the partial information as well as historical observations. In this setting, the objective is to derive an optimal threshold strategy $\hat n$ that maximizes the worst-case expected benefit rate, where the worst case is taken over all distributions from within this ambiguity set, i.e.,
    \begin{equation}\label{eq:dromodel}
        \max_{n \in \mathbb Z_+} \inf_{\mathbb P \in \mathcal P} \mathbb E_{\mathbb P}[c_n(\rhot)].
    \end{equation}
    Because the model optimizes the expected benefit rate in view of the worst-case distribution, it mitigates overfitting to the observed samples and helps improve the performance in out-of-sample circumstances. 
    
    In this paper, we study  the distributionally robust model from the perspective of an individual customer, a social optimizer, and a revenue maximizer. We first derive the results that hold for any generic ambiguity set $\mathcal P$. 
    
	\subsection{Individual Optimization}
	
	
	We determine a pure threshold strategy in which each arriving customer decides to join or not join the queue based on the observed queue length, independent of the strategy adopted by other customers. A newly arrived customer makes a decision (to join or not join) based on the net gain  $R-(i+1)C/\mu$, where $i$ is the number of people currently in the queue, and will join the queue if it is nonnegative. Note that net gain is  deterministic because it is independent of the random arrival rate. Thus,  the optimal joining threshold for any arriving customer is given by
	\begin{equation}
	\label{eq:io}
	\hat{n}_e = \left\lfloor{\frac{R{\mu}}{C}} \right\rfloor.
	\end{equation} 
	This result coincides with Naor's original result (i.e., $\hat n_e=n_e$) because the net gain of a newly arrived customer only depends on the current queue length and the service rate, which are all deterministic. On the other hand, as an individual optimizer, the customer can ignore the rates of future arrivals, because they will not affect the time to service. We also remark here that the individual threshold $n_e$ can be regarded as the maximal length of the strategic queue, beyond which  no newly arrived customer will ever enter the queue as the net gain becomes negative.

	\subsection{Social Optimization}
	We next analyze the distributionally robust threshold for a social optimizer. The social benefit rate for a realization of the traffic intensity $\rho$ and a fixed threshold $n$ is given by
	\begin{equation}
	\label{eq:social_rate}
	f_n(\rho)\coloneqq\left\{\begin{array}{lll}
	     &R \mu \frac{\rho (1-\rho^n)}{1-\rho^{n+1}} -C\left( \frac{ \rho}{1-\rho} -  \frac{ (n+1)\rho^{n+1}}{1-\rho^{n+1}}\right) \quad &\textup{if}\ \rho \neq 1\\
	     & R\mu \frac{n}{n+1}-C\frac{n}{2} &\textup{if}\ \rho=1. \\
	\end{array}\right.
	\end{equation}
	One can verify that $\lim_{\rho \rightarrow 1} R \mu \frac{\rho (1-\rho^n)}{1-\rho^{n+1}} -C\left( \frac{ \rho}{1-\rho} -  \frac{ (n+1)\rho^{n+1}}{1-\rho^{n+1}}\right)=R\mu \frac{n}{n+1}-C\frac{n}{2}$, which indicates that the function $f_n(\rho)$ is continuous in $\rho$. Here, the first term $\mu \tfrac{\rho (1-\rho^n)}{1-\rho^{n+1}}$  corresponds to the probability that an arriving customer joins, while the second term $\frac{\rho}{1-\rho} -  \frac{ (n+1)\rho^{n+1}}{1-\rho^{n+1}}$ represents the expected number of customers in the queue system \cite[Equation (2.3)]{hassin2003queue}. 
	
	The distributionally robust model determines an optimal threshold $\hat{n}_s$ that maximizes 
	the worst-case expected  social benefit rate  $ Z_s(n)$, i.e., $\hat n_s \in \argmax_{n \in \mathbb Z_+} Z_s(n)$, where
	\begin{equation}
	\label{eq:mean_so1}
	Z_s(n) :=\inf_{\mathbb P \in \mathcal P } \mathbb E_{\mathbb P} \left[f_n(\rhot) \right].
	\end{equation}
     We first investigate the relationship between the optimal thresholds $\hat n_e$ and $\hat n_s$.
     
	\begin{prop}\label{prop:so<ne}
	    There exists an optimal threshold of the social optimizer less than or equal to the optimal threshold of an individual customer, i.e., $\exists \hat{n}_s \ \st \ \hat{n}_s \leq \hat{n}_e.$
	\end{prop}
	
	\begin{proof}[Proof of Proposition \ref{prop:so<ne}]
	It is established in \cite[Equation 30]{naor1969regulation} that for any deterministic arrival rate $\lambda$ and service rate $\mu$, the optimal threshold from the perspective of a public goods regulator will be less than or equal to the optimal threshold of an individual customer. 
    Suppose that every optimal threshold that maximizes the worst-case expected social benefit rate is strictly greater than the optimal threshold of an individual customer, i.e., $\hat n_s > \hat n_e$ for all $\hat n_s \in \argmax_{n \in \mathbb Z_+} \inf_{\mathbb P \in \mathcal P } \mathbb E_{\mathbb P} \left[f_n(\rhot)\right]$.
    Then, based on our previous statement, for any fixed $\rho$ and any optimal $\hat n_s$, we have $n_s(\rho) \leq n_e=\hat n_e < \hat n_s$, where $n_s(\rho)$ is the corresponding optimal social threshold under the deterministic setting. Since $f_n(\rho)$ is discretely unimodal for any fixed $\rho$ \cite[Page 20]{naor1969regulation}, the relationship of the benefit rate can consequently be derived as
    \begin{equation*}
        f_{n_s(\rho)}(\rho) \geq f_{\hat n_e}(\rho) \geq f_{\hat n_s}(\rho) \quad \forall \rho \in \mathbb R_+.
    \end{equation*}
    Using this relationship, one can further establish that, for any ambiguity set $\mathcal P$,
    \begin{equation}\nonumber
        \inf_{\mathbb P \in \mathcal P} \mathbb E_{\mathbb P} \left[ f_{\hat n_e}(\rhot) \right] \geq \inf_{\mathbb P \in \mathcal P} \mathbb E_{\mathbb P} \left[ f_{\hat n_s}(\rhot) \right].
    \end{equation}
    Conversely, by the definition of $\hat n_s$, we also have $\inf_{\mathbb P \in \mathcal P} \mathbb E_{\mathbb P} \left[ f_{\hat n_e}(\rhot) \right] \leq \inf_{\mathbb P \in \mathcal P} \mathbb E_{\mathbb P} \left[ f_{\hat n_s}(\rhot) \right]$. This implies that $\inf_{\mathbb P \in \mathcal P} \mathbb E_{\mathbb P} \left[ f_{\hat n_e}(\rhot) \right] = \inf_{\mathbb P \in \mathcal P} \mathbb E_{\mathbb P} \left[ f_{\hat n_s}(\rhot) \right]$. Therefore, $\hat n_e$ is also an optimal threshold of the social optimization problem, which contradicts our previous assumption. This completes the proof.
	\end{proof}
	
	\subsection{Revenue Optimization}

    We now consider a profit-maximizing firm that aims to maximize its expected revenue rate by imposing a toll $t$ on every joining customer. In this setting, customers base their joining decision on this imposed toll $t$ and evaluate the service completion only by $R-t$. Recall that customers join the queue if and only if the expected net gain is nonnegative. Therefore, determining an optimal toll $t$ is equivalent to choosing a queue length threshold $n$ that maximizes the expected revenue rate, where $n=\left\lfloor{\frac{(R-t){\mu}}{C}} \right\rfloor.$
    The revenue rate for a realization of the traffic intensity and a fixed threshold $n$ is given by 
	\begin{equation}
	r_n(\rho)\coloneqq\left\{\begin{array}{lll}
	     \left(R\mu - Cn\right) \frac{\rho(1-\rho^n)}{1-\rho^{n+1}} &\textup{if}\ \rho \neq 1\\
	      \left(R\mu - Cn\right) \frac{n}{n+1} &\textup{if}\ \rho=1. \\
	\end{array}\right.
	\end{equation}
	One can show that $\lim_{\rho \rightarrow 1} \frac{\rho(1-\rho^n)}{1-\rho^{n+1}} = \frac{n}{n+1}$, which indicates that $f_n(\rho)$ is continuous.    The revenue rate function $r_n(\rho)$ can be rewritten as $\tfrac{R\mu-Cn}{\mu} \cdot \tfrac{\lambda(1-\rho^n)}{1-\rho^{n+1}} $, where $\tfrac{R\mu-Cn}{\mu}$ is the entrance fee for a given maximal queue length $n$, and $\frac{\lambda(1-\rho^n)}{1-\rho^{n+1}}$ is the expected number of customers joining the queue per unit time. 
	
	The distributionally robust model determines an optimal threshold $\hat{n}_r$ that 
    maximizes the worst-case expected  revenue rate  $Z_r(n)$, i.e., $\hat n_r \in \argmax_{n \in \mathbb Z_+} Z_r(n)$, where 
	\begin{equation}
	\label{eq:mean_rm}
	Z_r(n) \coloneqq \inf_{\mathbb P \in \mathcal P } \mathbb E_{\mathbb P} \left[r_n(\rhot) \right].
	\end{equation}
    Similarly, we first investigate the relationship between the optimal thresholds $\hat n_e$ and $\hat n_r$.
    
    \begin{prop}\label{prop:rm<ne}
	    There exists an optimal threshold of the revenue maximizer less than or equal to the optimal threshold of an individual customer, i.e., $\exists \hat{n}_r \ \st \ \hat{n}_r \leq \hat{n}_e.$
	\end{prop}

	\begin{proof}
	    The proof parallels that of Proposition \ref{prop:so<ne}---we omit for brevity.
	\end{proof}
	
	Up to now, we have presented the generic distributionally robust observable queue models for an individual customer, a social optimizer, and a revenue maximizer. However, we have not specified the ambiguity set for the social and revenue optimization problems. In the following sections, we will investigate different types of ambiguity sets and derive their tractable reformulations.
	
	\section{Distributionally Robust Strategic Queues with a MAD Ambiguity Set} \label{sec:DROM}
	In this section, we study the DRO model with a mean-absolute deviation (MAD) ambiguity set. Suppose the support~$[a,b]$, mean $m$ and MAD $d$ of the random parameter $\rhot$ are known to decision makers. Then we can construct an ambiguity set containing all possible distributions that are consistent with the partial information, defined as
	\begin{equation}\label{eq:ambiguity_set}
	    \mathcal P \coloneqq \{ \mathbb P \in \mathcal P_0([a,b]):\; \mathbb E_{\mathbb P} [\rhot] = m, \; \mathbb E_{\mathbb P} \left[|\rhot- m| \right] = d\}.
	\end{equation}
	We develop efficient solution schemes to find the optimal threshold strategies for a social optimizer and a revenue maximizer, given by $\hat{n}_s$ and $\hat{n}_r$, respectively, such that the worst-case expected benefit rates are maximized. In order to derive tractable reformulations for the distributionally robust models, we assume $m \in (a,b) $ and $d \in (0, \overline d)$, where $\overline d \coloneqq \frac{2(m-a)(b-m)}{b-a}$ is the largest possible mean-absolute deviation by any distribution with the given support and mean.

	\subsection{Social Optimization}
	To determine an optimal joining threshold for a social optimizer, we compute the worst-case expected social benefit rate $Z_s(n)$ for every $n\in \mathbb Z_+$ satisfying $1 \leq n \leq n_e$, and choose an $\hat n_s$ such that $\hat n_s \in \argmax_{n \in \mathbb Z_+} Z_s(n)$. To this end, we show how to compute the worst-case expected social benefit rate for a fixed $n$. Suppose the distribution mean and MAD of $\rhot$ are precisely known, then the worst-case expected social benefit rate is given by the optimal value of the moment problem
	
	\begin{equation}
	\label{eq:so_mean_abs_primal}
	\begin{array}{ccll}
Z_s(n)=    \vspace{1mm}&\displaystyle \inf_{\mathbb \nu \in \mathcal M_+}&\displaystyle \int_{\Xi} f_n(\rho) \nu (\rm d\rho) \\ 
	\vspace{1mm}&\st& \displaystyle \int_{\Xi} |\rho-m| \, \nu (\rm{d}\rho)=\mathit d\\
	\vspace{1mm}&& \displaystyle \int_{\Xi} \rho \, \nu(\rm d\rho)=\mathit{m}\\
	&& \displaystyle \int_{\Xi} \nu(\rm d\rho)= 1,
	\end{array} 
	\end{equation}
    where $\Xi:= [a,b]$ is the support of $\rhot$ and $\mathcal M_+$ denotes the set of all nonnegative measures. 
    The first and second constraints of \eqref{eq:so_mean_abs_primal} require the nonnegative measure's MAD and mean equals to $d$ and $m$, respectively, while the third constraint restricts the nonnegative measure to be a probability measure. The objective of the problem is to find a feasible probability measure that minimizes the expected social benefit rate. However, this semi-infinite linear optimization problem is hard to solve, because it searches for the best decision from an an infinite dimensional space of probability measures.
    To derive a tractable reformulation, we focus on the dual problem. We first define $F(\rho)\coloneqq \alpha|\rho-m| +\beta \rho + \gamma$,  and derive the dual problem as
    \begin{equation}
	\label{eq:so_mean_abs_dual}
	\begin{array}{ccll}
	&\displaystyle\sup_{\alpha,\beta,\gamma \in \RR}&\displaystyle \alpha d + \beta m + \gamma\\
	&\st&  F(\rho) \leq f_n(\rho) \quad \quad \forall \rho \in [a,b].
	\end{array} 
	\end{equation}
	Notice that $F(\rho)$ is a two-piece piecewise affine function majorized by $f_n(\rho)$. We know that if $f_n(\rho)$ is a piecewise affine function or a concave function, the semi-infinite constraint will reduce to a linear constraint since we only need to check the satisfaction of the constraint at points $\rho=a,m$ and $b$. However, the social benefit rate function is neither concave nor piecewise affine, making the problem difficult. To solve this optimization problem, we first investigate the properties of the social benefit rate function $f_n(\rho)$. Some of the proofs of this section are relegated to the Appendix \ref{sec:proof_sec3}

    \begin{lem}
    \label{lem:so_propertie}
    The social benefit rate function $f_n(\rho)$ has the following properties if $\frac{R\mu}{C} \geq n+1$:
    \begin{enumerate}
        \item $f_n(\rho)$ is strictly concave for $\rho \in [0,1]$.
        \item $f_n(\rho)$ is either concave increasing or unimodal for $\rho \in [0, \infty)$.
        \item The sign of the second derivative $f_n''(\rho)$ changes at most once over $[0, \infty)$.
    \end{enumerate}
    \end{lem}
    
    From Lemma \ref{lem:so_propertie} we know that the social benefit rate function has some nice properties. Specifically, the function is 
    either concave increasing or unimodal on the nonnegative axis, and when it is unimodal, the function changes from a concave function to a convex function at some point. The next lemma further asserts that the complementary slackness property holds for the primal and dual problems, which will later help us determine the worst-case distribution.
    \begin{lem}
    \label{lem:so_3points}
	The optimal values of the primal-dual pair \eqref{eq:so_mean_abs_primal} and~\eqref{eq:so_mean_abs_dual} coincide, and their optimal solutions $\nu^\star$ and $(\alpha^\star,\beta^\star,\gamma^\star)$, respectively, satisfy the  complementary slackness condition
	    \begin{equation}\nonumber
        \left(f_n(\rho)-\alpha^\star|\rho-m| -\beta^\star \rho - \gamma^\star\right)\nu^\star(\rm d \rho)=0\qquad \forall \rho\in[a,b]. 
    \end{equation}
    \end{lem}

	Combining Lemmas \ref{lem:so_propertie} and \ref{lem:so_3points}, we are ready to show that problem \eqref{eq:so_mean_abs_primal} can be solved analytically under certain conditions. Specifically, we divide this problem into three cases and derive an explicit expression of the worst-case distribution for each case.
    
	\begin{prop}
	\label{prop:so_mad}
		Assume $m\in [0,1]$ and $\frac{R\mu}{C} \geq n+1$. Let $(\rho_t, f_n(\rho_t))$ be the tangent point on $f_n$ for the line that passes through $(m,f_n(m))$. For any $n \geq 1$, we have one of the following three cases:
		\begin{enumerate}
		    \item If $f_n(b) + f'_n(b) (m-b) \geq f_n(m)$, then the extremal distribution that solves \eqref{eq:mean_so1} is a three-point distribution supported on $\rho_1 = a$, $\rho_2 = m$, $\rho_3=b$, with corresponding probabilities
		$$ p_1 = \frac{d}{2(m-a)}, \ p_2 = 1 - \frac{d}{2(m-a)} - \frac{d}{2(b-m)},\  p_3= \frac{d}{2(b-m)}.$$
		    \item If $f_n(b) + f'_n(b) (m-b) < f_n(m)$ and $d< d_0\coloneqq\frac{2(m-a)(\rho_{t}-m)}{\rho_t-a}$, then the extremal distribution is a three point distribution supported on $\rho_1 = a$, $\rho_2 = m$, $\rho_3=\rho_t$, with probabilities
	    $$ p_1 = \frac{d}{2(m-a)}, \ p_2 = 1 - \frac{d}{2(m-a)} - \frac{d}{2(\rho_t-m)},\  p_3= \frac{d}{2(\rho_t-m)}.$$
		    \item If $f_n(b) + f'_n(b) (m-b) < f_n(m)$ and $d\geq  d_0\coloneqq\frac{2(m-a)(\rho_{t}-m)}{\rho_t-a}$, then the extremal distribution is a two-point distribution supported on $\rho_1=a$, $\rho_2 = \frac{ad+2m(a-m)}{d+2(a-m)}$, with probabilities $$ p_1 = \frac{d}{2(m-a)}, \ p_2 = 1-\frac{d}{2(m-a)}.$$
		\end{enumerate}  
	\end{prop}

\begin{figure}[H]\label{fig:MAD_figure}
	\centering
	\subfigure[$f_n(m) \leq f_n(b) + f'_n(b) (m-b)$]{\includegraphics[width=.49\textwidth]{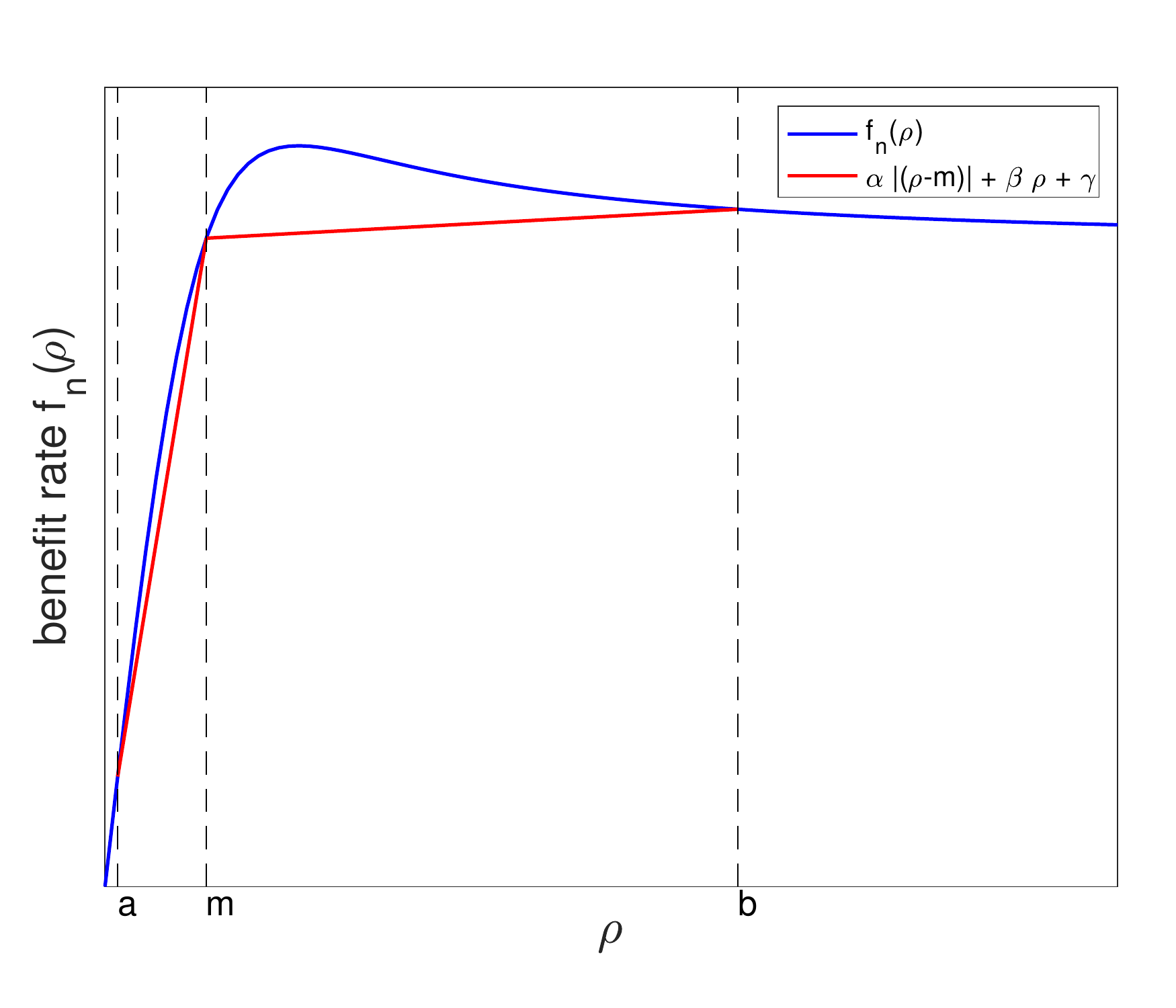}\label{fig:mad_impro_1}}	\subfigure[$f_n(m) \geq f_n(b) + f'_n(b) (m-b)$ and $d< d_0$]{\includegraphics[width=.49\textwidth]{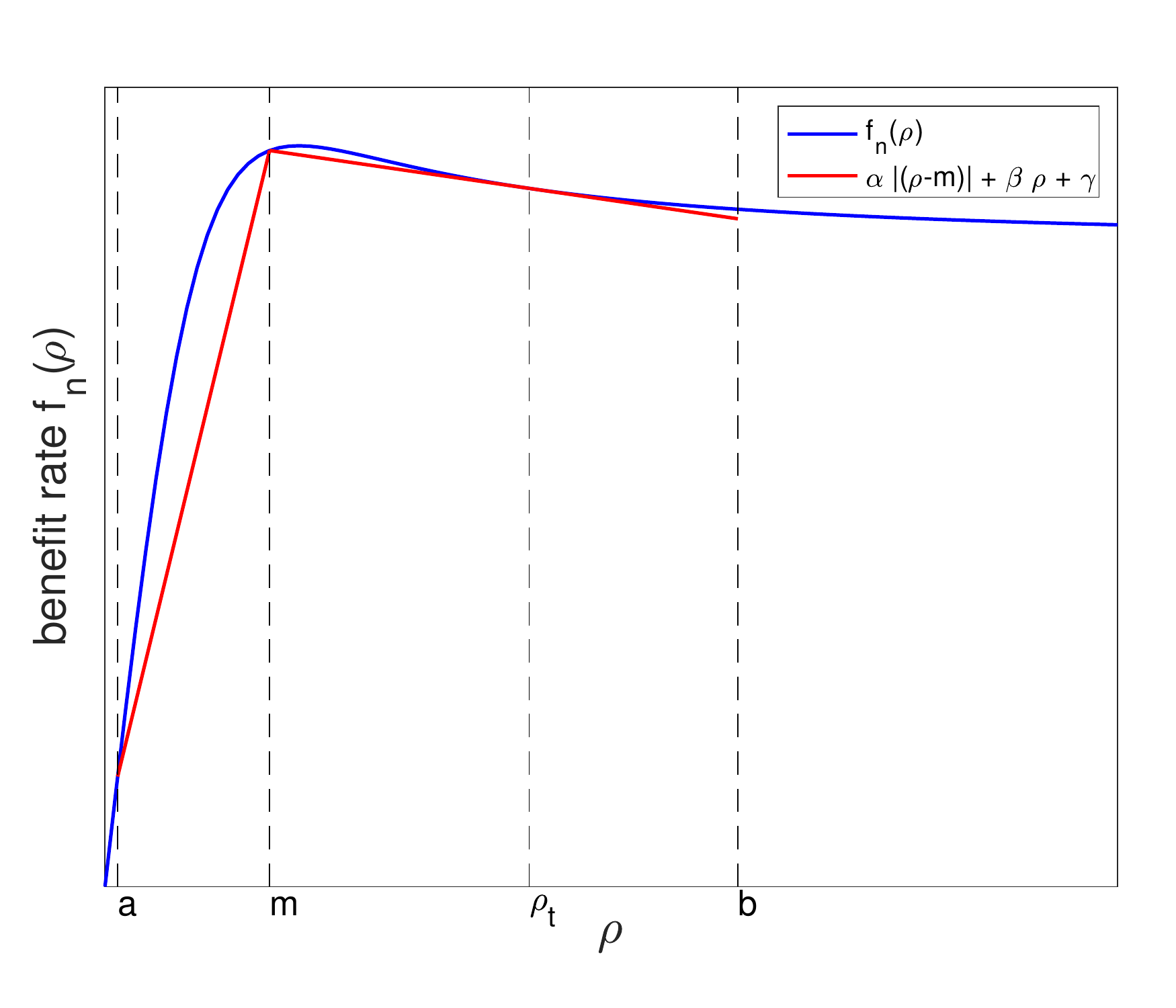}\label{fig:mad_impro_2}}
	\subfigure[$f_n(m) \geq f_n(b) + f'_n(b) (m-b)$ and $d \geq d_0$]{\includegraphics[width=.49\textwidth]{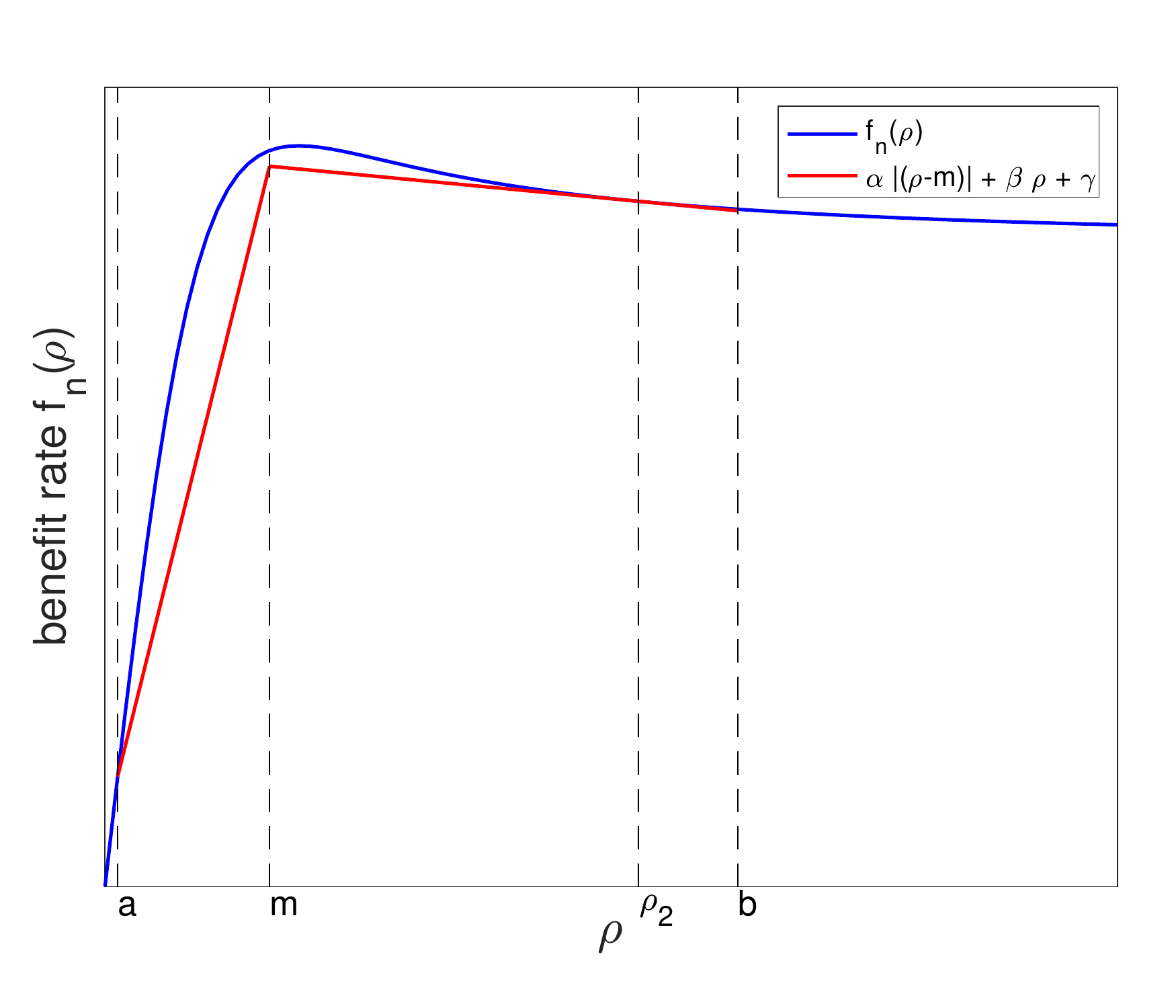}\label{fig:mad_impro_3}}
	\caption{\label{fig:oss} In Figure \ref{fig:mad_impro_1}, the optimal piecewise affine function is determined by points $(a,f_n(a))$, $(m,f_n(m))$, and $(b,f_n(b))$. In Figure \ref{fig:mad_impro_2}, the parameters satisfy $f_n(m) \geq f_n(b) + f'_n(b) (m-b)$ and $d< d_0$. Thus, the optimal two-piece piecewise affine function touches $f_n(\rho)$ at $(a,f_n(a))$, $(m,f_n(m))$, and $(\rho_t,f_n(\rho_t))$, where $(\rho_t,f_n(\rho_t)$ is the tangent point. In Figure \ref{fig:mad_impro_3}, $f_n(m) \geq f_n(b) + f'_n(b) (m-b)$ still holds, while $d\geq d_0$. In this case, the extremal distribution degenerates to a two-point distribution.}
\end{figure}

    Figure 1 depicts the optimal two-piece piecewise affine function described in Proposition \ref{prop:so_mad}. We remark that the tangent point $(\rho_t,f_n(\rho_t))$ in Figure \ref{fig:mad_impro_2} can be determined efficiently by the bisection method. Specifically, we set $[l,u]=[m,b]$ as the initial search interval for the algorithm. In each iteration, we compute the derivative at the midpoint $\rho=\frac{u+l}{2}$, and check whether it is the tangent point by calculating the difference between $f_n(m)$ and $f_n'(\frac{u+l}{2})(m-\frac{u+l}{2})+f_n(\frac{u+l}{2})$. If the difference is small enough, we 
   terminate the algorithm; otherwise, we set $u=\frac{u+l}{2}$ if the difference is positive or set  $l=\frac{u+l}{2}$ if the difference is negative, and then go back to the first step with the updated interval $[l,u]$.
   
   Proposition \ref{prop:so_mad} explicitly expresses the extremal distribution for parameters satisfying $m\leq 1$ and $\frac{R\mu}{C} \geq n+1$. Using this result, we can compute the worst-case expected social benefit rate $Z_s(n)$ efficiently.

\begin{thm}\label{thm:so_mad}
Assume $m\in [0,1]$ and $\frac{R\mu}{C} \geq n+1$. Let $(\rho_t, f_n(\rho_t))$ be the tangent point on $f_n(\rho)$ for the line that passes through  $(m,f_n(m))$. For any $n \geq 1$, we have the following three cases:
\begin{enumerate}
		    \item If $f_n(b) + f'_n(b) (m-b) \geq f_n(m)$, then 
		    \[Z_s(n)=\frac{d}{2(m-a)} f_n(a) + \left(1 - \frac{d}{2(m-a)} - \frac{d}{2(b-m)} \right) f_n(m)+ \frac{d}{2(b-m)} f_n(b).
		    \]
		    \item If $f_n(b) + f'_n(b) (m-b) < f_n(m)$ and $d< d_0\coloneqq\frac{2(m-a)(\rho_{t}-m)}{\rho_t-a}$, then 
		    \[Z_s(n)=\frac{d}{2(m-a)} f_n(a) + \left(1 - \frac{d}{2(m-a)} - \frac{d}{2(\rho_t-m)} \right) f_n(m)+ \frac{d}{2(\rho_t-m)} f_n(\rho_t).
		    \]
		    \item If $f_n(b) + f'_n(b) (m-b) < f_n(m)$ and $d\geq  d_0\coloneqq\frac{2(m-a)(\rho_{t}-m)}{\rho_t-a}$, then 
		    \[ Z_s(n)=\frac{d}{2(m-a)} f_n(a) + \left(1-\frac{d}{2(m-a)} \right) f_n\left(\frac{ad+2m(a-m)}{d+2(a-m)}\right).
		    \]
		\end{enumerate}  
\end{thm}
    Theorem \ref{thm:so_mad} enables us to solve the worst-case expectation problem analytically under certain conditions. However, for the more general case, we are unable to solve it in a closed form. In the following theorem, we show that the worst-case expectation problem admits a semidefinite programming reformulation that can be solved in polynomial time using standard off-the-shelf solvers, such as SDPT3~\cite{toh1999sdpt3} and MOSEK~\cite{mosek}.
	
		\begin{thm}
	\label{thm:MAD_so}
	For any $n \geq 1$, the worst-case expected social benefit rate $Z_s(n)$  coincides with the optimal value of the following semidefinite program.
		\begin{equation}
		\label{eq:so_sdp}
	    \begin{array}{llll}
	         &\sup \;\; &\alpha d + \beta m + \gamma &\\
	         &\st \;\; &\alpha,\beta,\gamma \in \mathbb R,  y, z \in \mathbb R^{n+3},  X, X' \in \mathbb S^{n+3}_+\\
	         && y_1 = R\mu - C - y_0 + y_{n+3}, \; y_2 = -R\mu - y_{n+3},& \\ 
	         && y_3, \cdots, y_n = 0, \;  y_{n+1} = - R\mu + C(n+1) - y_0,& \\ 
	         &&  y_{n+2} = R\mu - C n + y_0 - y_{n+3}, \;& \\ 
	         &&  y_{0} = \alpha m +\gamma , y_{n+3}= -\alpha + \beta \;&\\
	         &&\displaystyle \sum_{i + j = 2l - 1} x_{ij} = 0&\forall l \in [n+4]\\ 
	         &&\displaystyle \sum_{i + j = 2l} x_{ij} = \sum_{q=0}^{l} \sum_{r=q}^{n+3+q-l} y_r {r \choose q} {{n+3-r} \choose {l-q}} a^{r-q} m^q &  \forall l \in [n+4] \cup \{0\}\\
	         && z_1 = R\mu - C - z_0 + z_{n+3}, \; z_2 = -R\mu - z_{n+3}, &\\
	         && z_3, \cdots, z_n = 0, \;  z_{n+1} = - R\mu + C(n+1) - z_0, &\\
	         &&  z_{n+2} = R\mu - C n + z_0 - z_{n+3} \; &\\ 
	         && z_{0} = -\alpha m +\gamma , z_{n+3}= \alpha + \beta \; \\ 
	         && \displaystyle \sum_{i + j = 2l - 1} x'_{ij} = 0 & \forall l \in [n+4]\\
	         && \displaystyle \sum_{i + j = 2l} x'_{ij} = \sum_{q=0}^{l} \sum_{r=q}^{n+3+q-l} y'_r {r \choose q} {{n+3-r} \choose {l-q}} m^{r-q} b^q & \forall l \in [n+4] \cup \{0\}\\
	    \end{array}
	\end{equation}
	\end{thm}
	The proof of this theorem relies on the following lemma which expresses a univariate polynomial inequality in terms of semidefinite constraints. 
	\begin{lem}[Proposition 3.1(f) in \cite{bertsimas2005optimal}]
		\label{poly_lem_so1}		
		The polynomial $g(\rho) = \sum_{r = 0}^{k} y_r \rho^r$ satisfies $g(\rho) \geq 0$ for all $ \rho \in [a,b]$ if and only if there exists a positive semidefinite matrix $X = [x_{ij}]_{i, j = 0, \cdots, k}\in \mathbb S_+^{k+1}$, such that 
		\begin{align*}
		0 = &\sum_{i,j:i+j = 2l -1} x_{ij} && \forall l = 1, \cdots, k \\ 
		\nonumber \sum_{q=0}^{l} \sum_{r=q}^{k+q-l} y_r {r \choose q} {{k-r} \choose {l-q}} a^{r-q} b^q= & \sum_{i,j:i+j = 2l} x_{ij} && \forall l = 0, \cdots, k.
		\end{align*}		
	\end{lem}
	
	\begin{proof}[Proof of Theorem~\ref{thm:MAD_so}]
	Recall that the dual of $\inf_{\mathbb P \in \mathcal P}  \mathbb E_{\mathbb P} [f_n(\rhot)]$ for $\rhot$ supported on the interval $[a,b]$ is given by (cf. problem \eqref{eq:so_mean_abs_dual}):
	\begin{equation}
	\nonumber
	\begin{array}{ccll}
	&\displaystyle\sup_{\alpha,\beta,\gamma \in \RR}&\displaystyle \alpha d + \beta m + \gamma\\
	&\st& \displaystyle \alpha|\rho-m| +\beta \rho + \gamma \leq f_n(\rho) \quad \quad \forall \rho \in [a,b].
	\end{array} 
	\end{equation}
	We can deal with the semi-infinite constraint separately for the cases $\rho \leq m$ and $\rho \geq m$:
	\begin{equation}
	\nonumber
	\begin{array}{ccll}
	&\displaystyle\sup_{\alpha,\beta,\gamma \in \RR}&\displaystyle \alpha d + \beta m + \gamma\\
	&\st& \displaystyle \alpha(m-\rho) +\beta \rho + \gamma \leq f_n(\rho) \quad \quad \forall \rho \in [a,m]\\
	&& \displaystyle \alpha(\rho-m) +\beta \rho + \gamma \leq f_n(\rho) \quad \quad \forall \rho \in [m,b].
	\end{array} 
	\end{equation}
Substituting the definition of $f_n(\rho)$ in \eqref{eq:social_rate} and applying algebraic reductions yield the following polynomial inequalities:
	\begin{align}
	\footnotesize
	\label{eq:so_MAD_poly}
	\nonumber &-(\alpha m + \gamma) \rho^0 + (R\mu - C - \beta + \gamma+\alpha m + \alpha) \rho + (-R\mu -\alpha+ \beta) \rho^2 + (-R\mu + Cn + C + \alpha m + \gamma) \rho^{n+1}   \\ \nonumber  & \quad\quad+ (R\mu - Cn -\alpha m - \alpha+ \beta - \gamma) \rho^{n+2} +(\alpha-\beta) \rho^{n+3} \geq 0  \qquad  \,\forall \rho \in [a,m], \quad\textup { and }\\
	\nonumber &(\alpha m - \gamma) \rho^0 + (R\mu - C -\alpha m - \alpha- \beta + \gamma) \rho + (-R\mu +\alpha+ \beta) \rho^2 + (-R\mu + Cn + C -\alpha m+ \gamma) \rho^{n+1}   \\   & \quad\quad + (R\mu - Cn + \alpha m + \alpha+\beta - \gamma) \rho^{n+2} -(\alpha+\beta) \rho^{n+3} \geq 0 \qquad  \forall \rho \in [m,b].
	\end{align}
	The inequalities are of the form $g_1(\rho) = \sum_{r = 0}^{n+3} y_r \rho^r \geq 0$ for $\rho \in [a,m]$ and $g_2(\rho) = \sum_{r = 0}^{n+3} z_r \rho^r \geq 0$ for $\rho \in [m,b]$, where $y= (y_1,\ldots,y_{n+3})$ and $z=(z_1,\ldots,z_{n+3})$ represent the coefficients of the respective polynomial inequalities. We now invoke the result of Lemma \ref{poly_lem_so1} with $k = n+3$ to express the inequalities in \eqref{eq:so_MAD_poly} as semidefinite constraints. The resulting semidefinite problem is equivalent to the original problem, which completes the proof.
	\end{proof}
	
	   \begin{rem}
   	In this subsection, we present two results: Theorem \ref{thm:so_mad} provides a closed form solution under certain prerequisites, while Theorem \ref{thm:MAD_so} derives an SDP reformulation for the general cases.
   	It is worth noting that Theorem \ref{thm:so_mad} requires the parameters to satisfy $n \leq \frac{R\mu}{C} - 1$. By Proposition \ref{prop:so<ne}, there exists an optimal threshold $\hat{n}_s$ less than or equal to $\hat{n}_e$, i.e.,~$\exists \hat{n}_s \leq \hat {n}_e= \left\lfloor{\frac{R{\mu}}{C}} \right\rfloor$. Thus, for a strategic queue with maximum length $n \leq \left\lfloor{\frac{R{\mu}}{C}} \right\rfloor$ and  mean arrival rate $m \leq 1$, Theorem~\ref{thm:so_mad} can be applied to compute the worst-case expected social benefit rate for the first $\left\lfloor{\frac{R{\mu}}{C}} \right\rfloor-1$ cases.  This greatly speeds up to time to solve  \eqref{eq:mean_so1} since we only need to solve an SDP once for the remaining case $n=\left\lfloor{\frac{R{\mu}}{C}} \right\rfloor$. On the other hand, for a strategic queue with  mean arrival rate $m >1$, we cannot invoke Theorem \ref{thm:so_mad} anymore and need to solve an SDP for each $n$ satisfying $1 \leq n \leq \hat n_e$, $n \in \mathbb Z_+$.
   	\end{rem}
	
	\subsection{Revenue Optimization}
    
	
	
	To determine an optimal joining threshold for a revenue maximizer, we compute the worst-case expected revenue rate $Z_r(n)$ for every $n\in\mathbb Z_+$, $1 \leq n \leq n_e$, and choose an $\hat n_r$ such that $\hat n_r \in \argmax_{n \in \mathbb Z_+} \{Z_r(n)\}$. To this end, we show how to compute the worst-case expected revenue for each $n$. Suppose the mean and MAD of the uncertain parameter $\rhot$ are known, then the worst-case expected revenue rate is given by the following optimization problem:
	
	\begin{equation}
	\label{eq:rm_mean_abs_primal}
	\begin{array}{ccll}
    \vspace{1mm}Z_r(n)=&\displaystyle \inf_{\mathbb \nu \in \mathcal M_+}&\displaystyle \int_{\Xi} r_n(\rho) \nu (\rm d\rho) \\
	\vspace{1mm}&\st& \displaystyle \int_{\Xi} |\rho-m| \, \nu (\rm{d}\rho)=\mathit d\\
	\vspace{1mm}&& \displaystyle \int_{\Xi} \rho \, \nu(\rm d\rho)=\mathit{m}\\
	&& \displaystyle \int_{\Xi} \nu(\rm d\rho)= 1.
	\end{array} 
	\end{equation}
	To derive a tractable reformulation, we first investigate the property of the revenue rate function $r_n(\rho)$. 
	
	\begin{lem}\label{lem:rm_concave}
	    The revenue rate function $r_n(\rho)$ is concave for $\rho \in \mathbb R_+$.
	\end{lem}
	
	Equipped with Lemma \ref{lem:rm_concave}, we now show that the worst-case expectation problem \eqref{eq:rm_mean_abs_primal} admits a closed form solution.

	\begin{thm}
	\label{thm:rm_mean_abs}
		For any $n \geq 1$, the worst-case expected revenue rate can be derived as
		\[ Z_r(n)=\frac{d}{2(m-a)} f_n(a) + \left(1 - \frac{d}{2(m-a)} - \frac{d}{2(b-m)} \right) f_n(m)+ \frac{d}{2(b-m)} f_n(b).
		\]
	\end{thm}
	
	To prove this theorem, we invoke a classical result  that characterizes the worst-case distribution  from  the MAD ambiguity set for a concave loss function.
	\begin{lem} [Theorem 3 in \cite{Bental1972}]
	\label{lem:bental1972}
	Suppose $f(\rho)$ is a concave function and the ambiguity set is defined as $\mathcal P = \{ \mathbb P \in \mathcal P_0([a,b]):\; \mathbb E_{\mathbb P} [\rhot] = m, \; \mathbb E_{\mathbb P} \left[|\rhot- m| \right] = d\}$. The extremal distribution that solves $\inf_{\mathbb P \in \mathcal P} \mathbb E_{\mathbb P} [f(\rhot)]$
	is a three point distribution supported on $\rho_1=a$, $\rho_2 = m$, $\rho_3 = b$  with probabilities
	\begin{equation}
	    \label{eq:three_points_dist}
	    p_1 = \frac{d}{2(m-a)},\ p_2 = 1 - \frac{d}{2(m-a)} - \frac{d}{2(b-m)},\ p_3= \frac{d}{2(b-m)}.
	\end{equation}
	\end{lem}
	\begin{proof}[Proof of Theorem \ref{thm:rm_mean_abs}]
	     From Lemma \ref{lem:rm_concave}, the revenue rate function $r_n(\rho)$ is concave. Therefore, applying Lemma \ref{lem:bental1972} yields the result.
	\end{proof}
	
\section{Extension to the Data-Driven Setting}\label{sec:data-driven-mad}

In this section, we design a distributionally robust model  using a purely data-driven ambiguity set constructed from historical samples. As we observed in the previous section, distributionally robust models with a moment ambiguity set necessitate decision makers to have access to precise values of  the mean, variance, or MAD of the true unknown distribution, which may not be realistic in practice. Decision makers usually construct such moment ambiguity sets by plugging in the point estimates generated from the historical samples. However, it is rarely the case that one can be entirely confident in these empirical estimators. For example, when the sample size is small, these empirical estimators might be far away from the true values; furthermore, some estimators, such as the empirical MAD, are even biased. In order to mitigate the adverse effects of the estimation errors, we develop a distributionally robust model with a data-driven MAD ambiguity set. 

Unlike the setting in the previous section, here we assume queue system managers only have access to $N$ independent and identically distributed samples of the traffic intensity given by $\{\hat{\rho}_i\}_{i \in [N]}$, where $\hat \rho_i = \hat \lambda_i/\mu$. 
	Suppose the true mean and MAD of the underlying distribution are unknown and belong to two confidence intervals $\mathcal T=[m_l,m_u]$ and $\mathcal D=[d_l,d_u]$ with high probabilities, where $\mathcal T$ and $\mathcal D$ are constructed using the samples. Then the proposed data-driven distributionally robust model is formulated as
	\begin{equation}\label{eq:ddmad}
	    \max_{n \in \mathbb Z_+} \inf_{m \in \mathcal T, d \in \mathcal D} \inf_{ \mathbb P \in \mathcal P} \mathbb E_{\mathbb P} [c_n(\rhot)],
	\end{equation}
	where $\mathcal P$ is the primitive MAD ambiguity set defined in \eqref{eq:ambiguity_set}. By optimizing in view of the worst-case mean and MAD, the model provides another layer of robustification  against errors from the empirical estimators.
	
	Observe that the inner two-layer infimum problem can be rewritten as
	\begin{equation}\label{eq:ddmad_inf}
	    \overline{Z}(n)\coloneqq \inf_{ \mathbb P \in \mathcal P'_N} \mathbb E_{\mathbb P} [c_n(\rhot)],
	\end{equation}
	where the modified data-driven ambiguity set is defined as
	\begin{equation}\label{eq:data-driven-ambiguityset}
	    \mathcal P'_N= \{ \mathbb P \in \mathcal P_0([a,b]):\; m_l \leq \mathbb E_{\mathbb P} [\rhot] \leq m_u, \; d_l \leq \mathbb E_{\mathbb P} \left[|\rhot- m| \right] \leq d_u\}.
	\end{equation}
	 Therefore,  the results of Propositions \ref{prop:so<ne} and \ref{prop:rm<ne} still hold, and we can obtain the optimal value of \eqref{eq:ddmad} by solving $\overline{Z}(n)$ for each $n\in \mathbb Z_+$ satisfying $1 \leq n \leq n_e$ and select the one with the largest objective value. 
	 
	 We now derive the reformulations for the worst-case expected social benefit and revenue rates. To this end, we define the worst-case expected social benefit rate with the data-driven MAD ambiguity set by
	 \[ \overline{Z}_s(n)\coloneqq \inf_{ \mathbb P \in \mathcal P'_N} \mathbb E_{\mathbb P} [f_n(\rhot)],
	 \] and the worst-case expected revenue rate with the data-driven MAD ambiguity set by
	 \[ \overline{Z}_r(n)\coloneqq \inf_{ \mathbb P \in \mathcal P'_N} \mathbb E_{\mathbb P} [r_n(\rhot)].
	 \]
	 The next theorem presents the reformulation of the worst-case expected social benefit rate. We relegate the proofs of this section to the Appendix \ref{sec:proof_sec4}.

	\begin{thm}\label{thm:data-driven-social}
	For any $n \geq 1$, the worst-case expected social benefit rate $\overline{Z}_s(n)$ coincides with the optimal value of the following semidefinite problem:
	\begin{equation*}
	   \begin{array}{llll}
	         &\sup \;\; &\gamma+\theta_1 d_l - \theta_2 d_u + \theta_3 m_l - \theta_4 m_u\\
	         &\st \;\; &\gamma \in \mathbb R, \theta_1,\theta_2,\theta_3,\theta_4 \in \mathbb R_+,  y, z \in \mathbb R^{n+3},  X, X' \in \mathbb S^{n+3}_+\\
	         && y_1 = R\mu - C - y_0 + y_{n+3}, \; y_2 = -R\mu - y_{n+3},& \\ 
	         && y_3, \cdots, y_n = 0, \;  y_{n+1} = - R\mu + C(n+1) - y_0,& \\ 
	         &&  y_{n+2} = R\mu - C n + y_0 - y_{n+3}, \;& \\ 
	         &&  y_{0} = (\theta_1-\theta_2) \hat m +\gamma , y_{n+3}= -\theta_1+\theta_2 + \theta_3-\theta_4 \;&\\
	         &&\displaystyle \sum_{i + j = 2l - 1} x_{ij} = 0&\forall l \in [n+4]\\ 
	         &&\displaystyle \sum_{i + j = 2l} x_{ij} = \sum_{q=0}^{l} \sum_{r=q}^{n+3+q-l} y_r {r \choose q} {{n+3-r} \choose {l-q}} a^{r-q} \hat m^q &  \forall l \in [n+4] \cup \{0\}\\
	         && z_1 = R\mu - C - z_0 + z_{n+3}, \; z_2 = -R\mu - z_{n+3}, &\\
	         && z_3, \cdots, z_n = 0, \;  z_{n+1} = - R\mu + C(n+1) - z_0, &\\
	         &&  z_{n+2} = R\mu - C n + z_0 - z_{n+3} \; &\\ 
	         && z_{0} = -(\theta_1-\theta_2) \hat m +\gamma , z_{n+3}= \theta_1-\theta_2 + \theta_3-\theta_4 \; \\ 
	         && \displaystyle \sum_{i + j = 2l - 1} x'_{ij} = 0 & \forall l \in [n+4]\\
	         && \displaystyle \sum_{i + j = 2l} x'_{ij} = \sum_{q=0}^{l} \sum_{r=q}^{n+3+q-l} y'_r {r \choose q} {{n+3-r} \choose {l-q}} \hat m^{r-q} b^q & \forall l \in [n+4] \cup \{0\}\\
	    \end{array}
	\end{equation*}
	\end{thm}

	Note that when $d_l=d_u$ and $m_l=m_u$, setting $\alpha=\theta_1-\theta_2$ and $\beta=\theta_3 -\theta_4$ recovers the dual problem \eqref{eq:so_mean_abs_dual} in view of the primitive MAD ambiguity set, which indicates the case when we have absolute trust on the mean and MAD estimators. 
	
	The next theorem presents the reformulation of the worst-case expected revenue rate.

	\begin{thm}\label{thm:data-driven-revenue}
	For any $n \geq 1$, the worst-case expected revenue rate $\overline{Z}_r(n)$ is equal to the optimal value of the following linear problem:
	\begin{equation*}
	\begin{array}{ccll}
	&\displaystyle\sup_{\theta \in \RR^4_+,\gamma \in \RR}&\displaystyle \gamma+\theta_1 d_l - \theta_2 d_u + \theta_3 m_l - \theta_4 m_u\\
	&\st& \displaystyle (\theta_1 - \theta_2)|a-\hat m| +(\theta_3-\theta_4) a + \gamma \leq r_n(a)\\
	&& \displaystyle (\theta_3-\theta_4) \hat m + \gamma \leq r_n(\hat m)\\
	&& \displaystyle (\theta_1 - \theta_2)|b-\hat m| +(\theta_3-\theta_4) b + \gamma \leq r_n(b).
	\end{array} 
	\end{equation*} 
	\end{thm}
	
	Theorems \ref{thm:data-driven-social} and \ref{thm:data-driven-revenue} provide tractable reformulations for the social and revenue optimization problems. An advantage of the data-driven model is that it can offer attractive finite-sample guarantees. 
	Compared with the original MAD ambiguity set that imposes unique mean and MAD, the data-driven MAD ambiguity set allows these parameters to vary within the confidence intervals. In this way, we can assure that the set contains the true underlying distribution with a high probability, which immediately generates out-of-sample performance guarantees for the solution.


	\begin{thm}\label{thm:finite-sample}
	Let $\{\hat{\rho}_i\}_{i \in [N]}$ be a set of $N$ samples generated independently at random from $\mathbb P^\star$ and~$v^\star$ denote the optimal value of \eqref{eq:ddmad}. By setting 
	\begin{align}\label{eq:confidence_interval}
	    \nonumber&\mathcal T=\left[\hat m - (b-a) \sqrt{\frac{\log 4/\delta}{2N}}, \hat m + (b-a) \sqrt{\frac{\log 4/\delta}{2N}}\;\right]\\
	    &\mathcal D=\left[\hat d - (b-a) \sqrt{\frac{9\log 4/\delta}{2N}}, \hat d + (b-a) \sqrt{\frac{9\log 4/\delta}{2N}}\;\right],
	\end{align}
    we have
    \[\textup{Prob} \left(v^\star \leq  \mathbb E_{\mathbb P^\star} [c_{\hat n}(\rho)] \right) \geq 1-\delta ,
    \]
    where $\hat n$ is the optimal threshold obtained from \eqref{eq:ddmad}.
	\end{thm}
	
	


\begin{proof}
    The error of the empirical MAD estimate is given by
    \begin{align*}
    &\nonumber\left| \frac{1}{N} \displaystyle \sum_{i=1}^N |\hat \rho_i-\hat m| - \mathbb E \left[|\tilde \rho -\hat m |\right] \right| \\
    =& \max \left\{ \frac{1}{N} \displaystyle \sum_{i=1}^N |\hat \rho_i-\hat m| - \mathbb E \left[|\tilde \rho -\hat m |\right], -\frac{1}{N} \displaystyle \sum_{i=1}^N |\hat \rho_i-\hat m| + \mathbb E \left[|\tilde \rho -\hat m |\right] \right\}.
    \end{align*}
    We upper bound both terms inside the max operator. The first term is bounded by
    \begin{align*}
    \nonumber\frac{1}{N} \displaystyle \sum_{i=1}^N |\hat \rho_i-\hat m| - \mathbb E \left[|\tilde \rho -\hat m |\right] \leq \;& \frac{1}{N} \displaystyle \sum_{i=1}^N |\hat \rho_i-\hat m|- \mathbb E \left[|\left|\rhot - m| - |\hat m -m| \right| \right]\\
    \leq \; &\frac{1}{N} \displaystyle \sum_{i=1}^N |\hat \rho_i-\hat m|- \mathbb E \left[|\rhot - m| - |\hat m -m|  \right] \\ 
    \leq \; & \left|\frac{1}{N} \displaystyle \sum_{i=1}^N |\hat \rho_i-\hat m| - \mathbb E \left[|\tilde \rho - m |\right] \right|+ |\hat m -m|,
    \end{align*}
    where the second inequality follows from reverse triangle inequality. Meanwhile, the second term is bounded by
    \begin{align*}
        -\frac{1}{N} \displaystyle \sum_{i=1}^N |\hat \rho_i-\hat m| + \mathbb E \left[|\tilde \rho -\hat m |\right]  \leq \;& -\frac{1}{N} \displaystyle \sum_{i=1}^N |\hat \rho_i-\hat m| + \mathbb E \left[|\tilde \rho -m|+|\hat m - m|\right]\\
        \leq \;& \left|\frac{1}{N} \displaystyle \sum_{i=1}^N |\hat \rho_i-\hat m| - \mathbb E \left[|\tilde \rho - m |\right] \right|+ |\hat m -m|.
    \end{align*}
    Since both of these two terms have the same upper bound, we have
    \begin{equation*}
        \left| \frac{1}{N} \displaystyle \sum_{i=1}^N |\hat \rho_i-\hat m| - \mathbb E \left[|\tilde \rho -\hat m |\right] \right| \leq \;\left|\frac{1}{N} \displaystyle \sum_{i=1}^N |\hat \rho_i-\hat m| - \mathbb E \left[|\tilde \rho - m |\right] \right|+ |\hat m -m|.
    \end{equation*}
    As $\mathbb E[\hat m]= \mathbb E[m]$ is an unbiased estimator, we can invoke the Hoeffding's inequality to derive a confidence interval for the second term. However, the empirical MAD is biased, i.e., $\mathbb E [\frac{1}{N} \sum_{i=1}^N |\rho_i-\hat m|] \neq \mathbb E \left[|\rho -m |\right]$---making the Hoeffding's inequality not applicable. To derive a confidence interval for this term, we rewrite it as
    \begin{align*}\label{eq:abs_max}
    &\nonumber\left| \frac{1}{N} \displaystyle \sum_{i=1}^N |\hat \rho_i-\hat m| - \mathbb E \left[|\tilde \rho -m |\right] \right| \\
    =& \max \left\{ \frac{1}{N} \displaystyle \sum_{i=1}^N |\hat \rho_i-\hat m| - \mathbb E \left[|\tilde \rho -m |\right], -\frac{1}{N} \displaystyle \sum_{i=1}^N |\hat \rho_i-\hat m| + \mathbb E \left[|\tilde \rho -m |\right] \right\}.
    \end{align*}
    We further upper bound the two terms inside the max operator. For the first term, we have
    \begin{align*}
    \nonumber\frac{1}{N} \displaystyle \sum_{i=1}^N |\hat \rho_i-\hat m| - \mathbb E \left[|\tilde \rho -m |\right] \leq \;& \nonumber\frac{1}{N} \displaystyle \sum_{i=1}^N |\hat \rho_i-m| + |m-\hat m| - \mathbb E \left[|\tilde \rho -m |\right]\\
    \leq \; & \left| \frac{1}{N} \displaystyle \sum_{i=1}^N |\hat \rho_i-m| - \mathbb E \left[|\tilde \rho -m |\right] \right| + |m-\hat m|. 
    \end{align*}
    For the second term, applying reverse triangle inequality yields
    \begin{align*}
     \nonumber \mathbb E \left[|\tilde \rho -m |\right] - \displaystyle \frac{1}{N}\sum_{i=1}^N |\hat \rho_i-\hat m| 
    \leq &\;  \nonumber  \mathbb E \left[|\tilde \rho -m |\right] - \frac{1}{N} \displaystyle \sum_{i=1}^N \left||\hat \rho_i-m| - |\hat m - m|\right|\\
    \leq & \; \nonumber \mathbb E \left[|\tilde \rho -m |\right] - \displaystyle \frac{1}{N} \sum_{i=1}^N |\hat \rho_i-m|  + |\hat m - m|\\    
    \leq & \;\left|\mathbb E \left[|\tilde \rho -m |\right] - \displaystyle \frac{1}{N} \sum_{i=1}^N |\hat \rho_i-m| \right| + |\hat m - m|.
    \end{align*}
Thus, we have 
    \begin{equation*}
       \left| \frac{1}{N} \displaystyle \sum_{i=1}^N |\hat \rho_i-\hat m| - \mathbb E \left[|\tilde \rho -m |\right] \right| \leq  \left|\mathbb E \left[|\tilde \rho -m |\right] - \displaystyle \frac{1}{N} \sum_{i=1}^N |\hat \rho_i-m| \right| + 2|\hat m - m|.
    \end{equation*}
    Since both of these two terms are unbiased, we can apply the Hoeffding's inequality and obtain
    \begin{align}
        &\nonumber\textup{Prob} \left(\left|\mathbb E \left[|\tilde \rho -m |\right] - \displaystyle \frac{1}{N} \sum_{i=1}^N |\hat \rho_i-m| \right| \geq r_1\right) \leq 2 \exp \left(-\frac{2Nr_1^2}{(b-a)^2}\right) \quad \textup{and}\\
        &\nonumber\label{eq:mean_hoeffding}\textup{Prob} \left(|\hat m - m| \geq r_2\right) \leq 2 \exp \left(-\frac{2Nr_2^2}{(b-a)^2}\right).
    \end{align}
    By applying the union bound and setting $r_1=r_2=r/3$, we arrive at the desired confidence intervals that the true mean $m$ and MAD $d$ satisfy
   \begin{equation*}
    \begin{array}{lll}
         \hat m - (b-a) \sqrt{\frac{\log 4/\delta}{2N}} \leq& m & \leq \hat m + (b-a) \sqrt{\frac{\log 4/\delta}{2N}}\\
         \hat d - (b-a) \sqrt{\frac{9\log 4/\delta}{2N}} \leq& d & \leq \hat d + (b-a) \sqrt{\frac{9\log 4/\delta}{2N}}
    \end{array}
    \end{equation*}
    with probability at least $1-\delta$. 
Therefore, by setting the confidence interval $\mathcal T$ and $\mathcal D$ as in \eqref{eq:confidence_interval}, we have
\[ \textup{Prob} \left(\mathcal P'_N \ni \mathbb P^\star \right) \geq 1-\delta,
\]where $\mathcal P'_N$ is the data-driven ambiguity set \eqref{eq:data-driven-ambiguityset} constructed by $N$ random  samples 
drawn from the underlying distribution $\mathbb P^\star$.
As $\nu^\star$ is defined by 
$\nu^\star \coloneqq \inf_{ \mathbb P \in \mathcal P'_N} \mathbb E_{\mathbb P} [c_n(\rhot)]$ and the probability of $\mathcal P'_N$ contains the true distribution~$\mathbb P^\star$ is greater than $1-\delta$, we have
\[\textup{Prob} \left(v^\star \leq  \mathbb E_{\mathbb P^\star} [c_{\hat n}(\rho)] \right) \geq 1-\delta ,
\]
which completes the proof.
\end{proof}
The theorem establishes  that, with judicious choices of the confidence interval lengths, the optimal value of the data-driven DRO model $v^\star$ provides  a high confidence lower bound on the expected benefit rate of the robust solution $\hat n$ under the true underlying distribution $\mathbb P^\star$. 

    
\begin{rem} An avid reader may be interested in employing the popular Wasserstein DRO model in the data-driven setting. Indeed, the model has been widely adopted because it can generate asymptotically consistent solutions and offer similarly attractive finite-sample guarantees. Unfortunately, 
the reformulation of this data-driven DRO model involves $\mathcal O(N)$ semidefinite constraints, which makes the problem computationally intensive. For readers who are interested in the use of Wasserstein ambiguity set, we provide a detailed discussion in Appendix B.
\end{rem}

\section{Numerical Experiment}\label{sec:numerical}

In this section we present the numerical experiments and examine the performance of different DRO policies.
All optimization problems are implemented in MATLAB  and solved by SDPT3~\cite{toh1999sdpt3} via the YALMIP interface~\cite{lofberg2004yalmip}. The experiments are run on a 2.2GHz Intel Core i7 CPU laptop with 8GB RAM.

We assess the out-of-sample performance of the data-driven policies for a social optimizer and a revenue maximizer through a fair out-of-sample experiment. 
We assume we have access to  $N$ independent samples $\{\hat{\rho}_i\}_{i \in [N]}$ of the traffic intensity drawn from the true underlying distribution~$\mathbb P^\star$, and we construct three ambiguity sets: an empirical MAD ambiguity set, a data-driven MAD (DD-MAD) ambiguity set and a Wasserstein ambiguity set.  
The empirical MAD ambiguity set is defined in \eqref{eq:ambiguity_set}, where we directly substitute the empirical mean and MAD  for $m$ and $d$, respectively. The DD-MAD ambiguity set is defined in \eqref{eq:ddmad} where, rather than carelessly plugging in the empirical estimators, we construct a confidence interval around the empirical mean and MAD. The Wasserstein ambiguity set~\cite{esfahani2017data,gao2016distributionally} is a popular data-driven ambiguity set. However, its complexity scales with the number of samples, making the problem computationally intensive with large sample sizes. We derive the reformulation of the Wasserstein model in Appendix \ref{sec:wasser}.
Once we constructed the ambiguity sets, we then proceed to compute the distributionally robust thresholds that maximize the worst-case expected benefit rate under these ambiguity sets. Finally, we compare the three solutions in a fair out-of-sample experiment relative to the  sample average approximation (SAA) method, which  n\"aively assumes that the empirical distribution generated from the $N$ samples is the true underlying distribution.

We conduct the out-of-sample trials for datasets containing $N = 2, 4, \dots, 10, 20, 40 \dots,100$ independent samples.   
We assume the arrival rate is generated by $\lambda=2 \tilde b$, where $\tilde b\sim Beta(0.1,0.5)$.
In each trial, we draw $N$ independent training samples and obtain $\{\hat{\rho}_i\}_{i \in [N]}$ from $\mathbb P^\star$. We then compute the optimal thresholds  $ \hat n_{d}$, $ \hat n_{dd}$ and $\hat n_{w}$ for the MAD, DD-MAD, and Wasserstein DRO models, respectively. We also compute the SAA threshold $\hat n_{SAA}$ by solving the sample average approximation model.  Based on the scaling rates derived in Theorem \ref{thm:finite-sample}  and \cite[Theorem 3.4]{esfahani2017data}, the size of the confidence intervals in \eqref{eq:ddmad} is set to be $C_1/\sqrt{N}$  and the Wasserstein radius is set to be $C_2/\sqrt{N}$, where $C_1$ and $C_2$ are chosen from the set $\{5,1,0.5,0.1,0.05,0.01 \}$ using a $\emph{k-fold cross validation}$ procedure. Specifically, we partition the in-sample data $\{\hat{\rho}_i\}_{i \in [N]}$ into $k=\min\{N,5\}$ folds and repeat the following procedure for each fold: the $i$-th fold is taken as a validation dataset and the remaining $k-1$ folds are merged to be a subtraining set. We repeat this process for each fold and choose the interval length that performs best in average.
The out-of-sample expected benefit/revenue rate $\mathbb E_{\mathbb P^\star}[c_{\hat n}(\rho)]$ for each of the strategies is then estimated  at high accuracy using 10,000 test samples from $\mathbb P^\star$.

\begin{figure}
	\centering
	\subfigure[Expected Value - Social Optimizer]{\includegraphics[width=.49\textwidth]{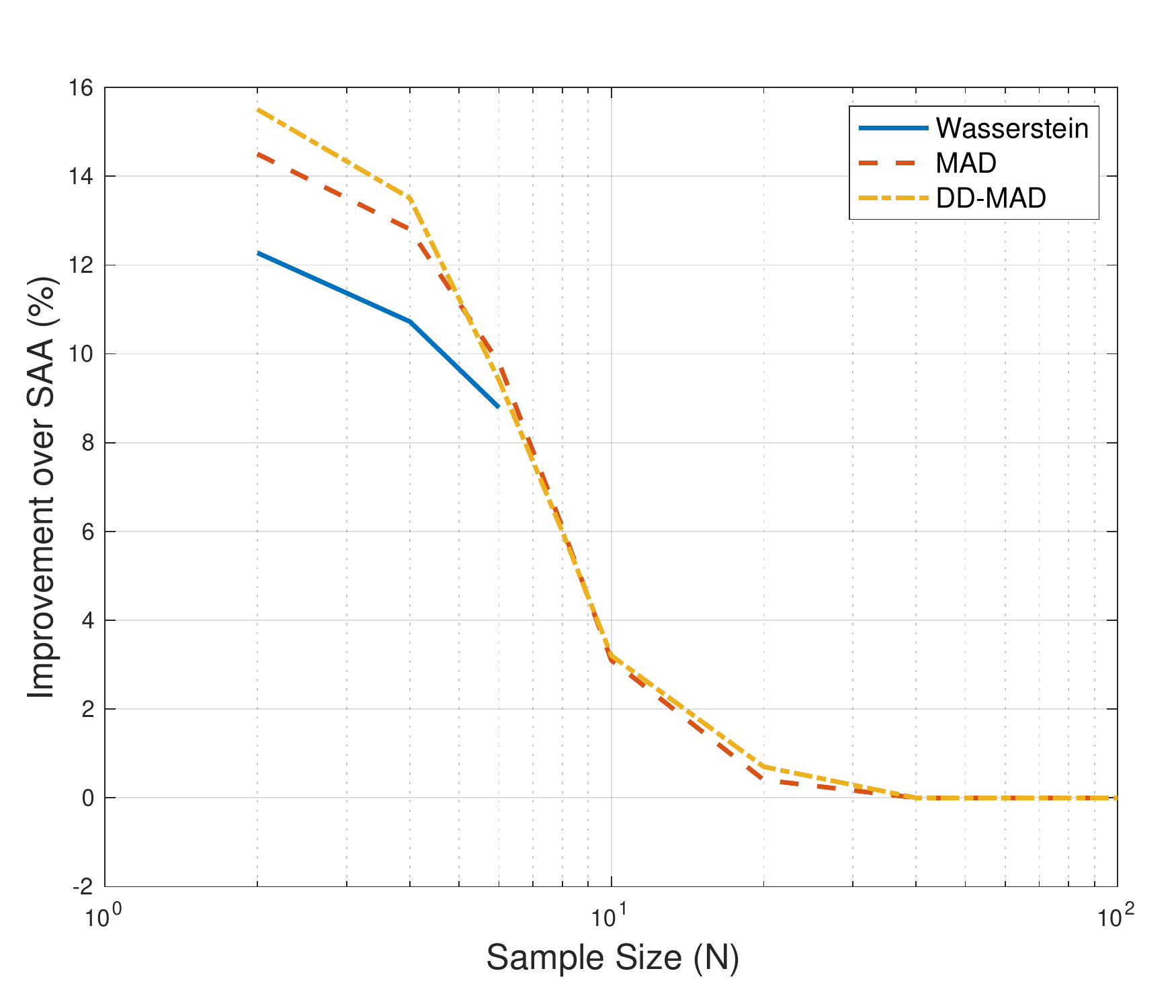}\label{fig:so_oss_impro}}
	\subfigure[95th Percentile - Social Optimizer]{\includegraphics[width=.49\textwidth]{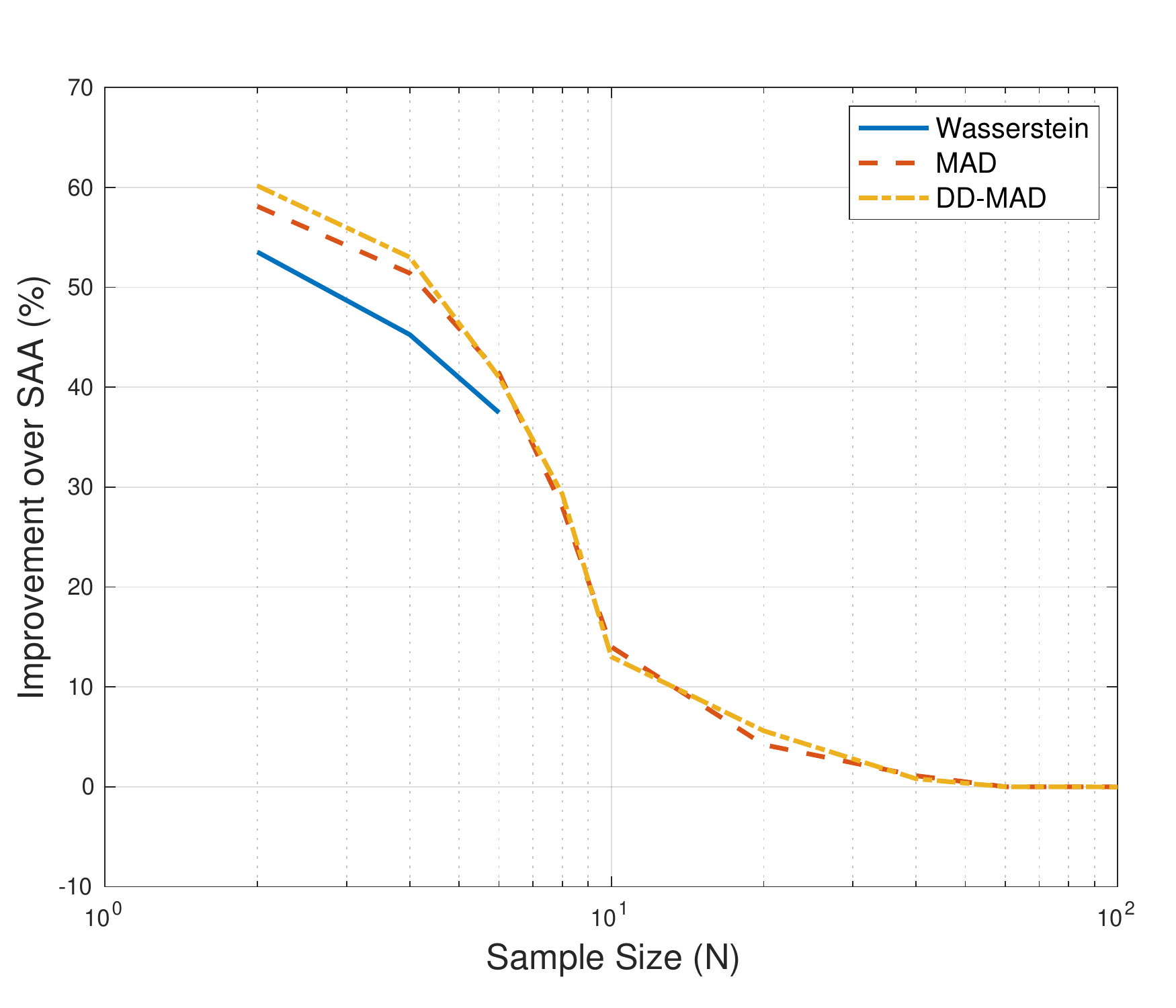}\label{fig:so_oss_impro}}
	\subfigure[Expected Value - Revenue Maximizer]{\includegraphics[width=.49\textwidth]{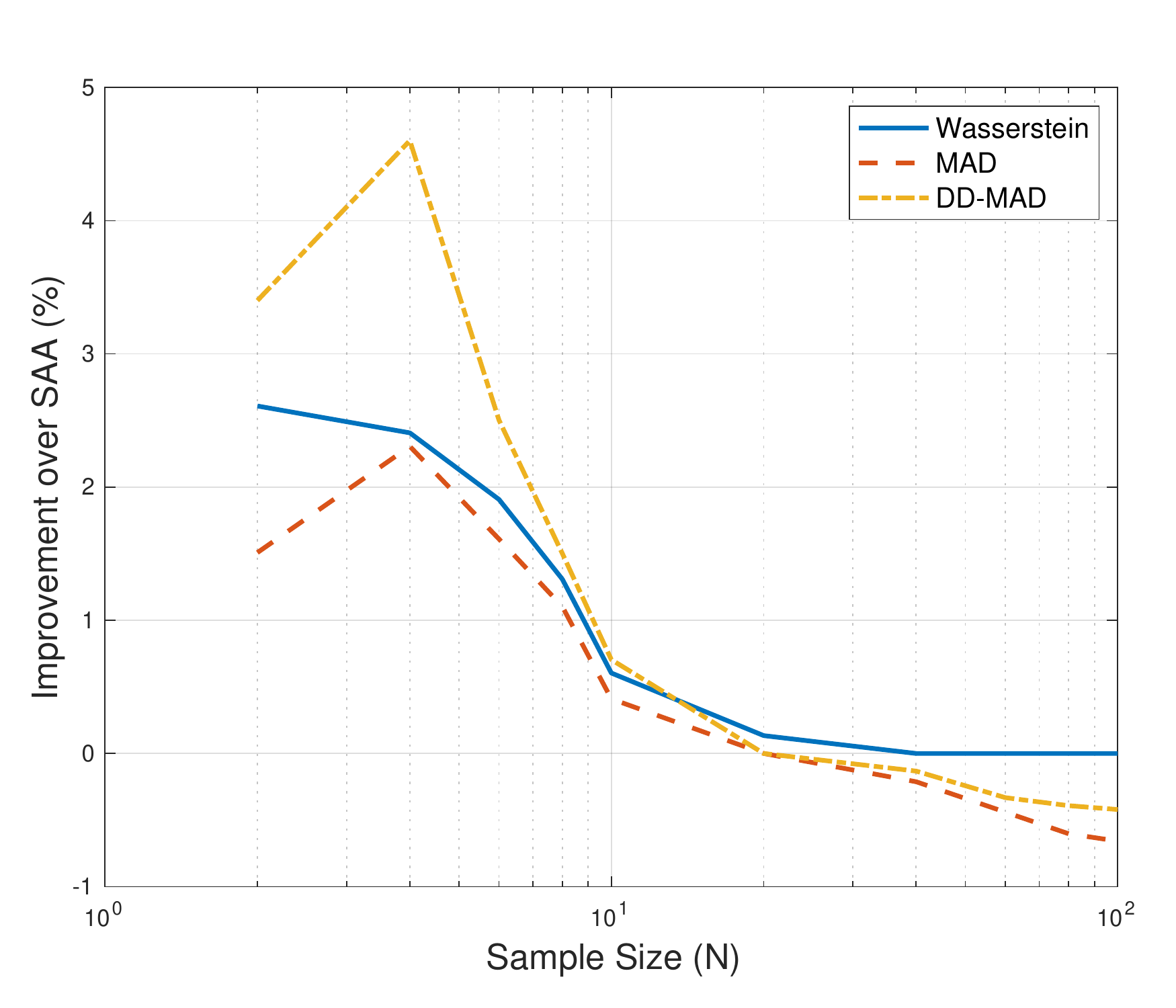}\label{fig:so_oss_impro}}
	\subfigure[95th Percentile - Revenue Maximizer]{\includegraphics[width=.49\textwidth]{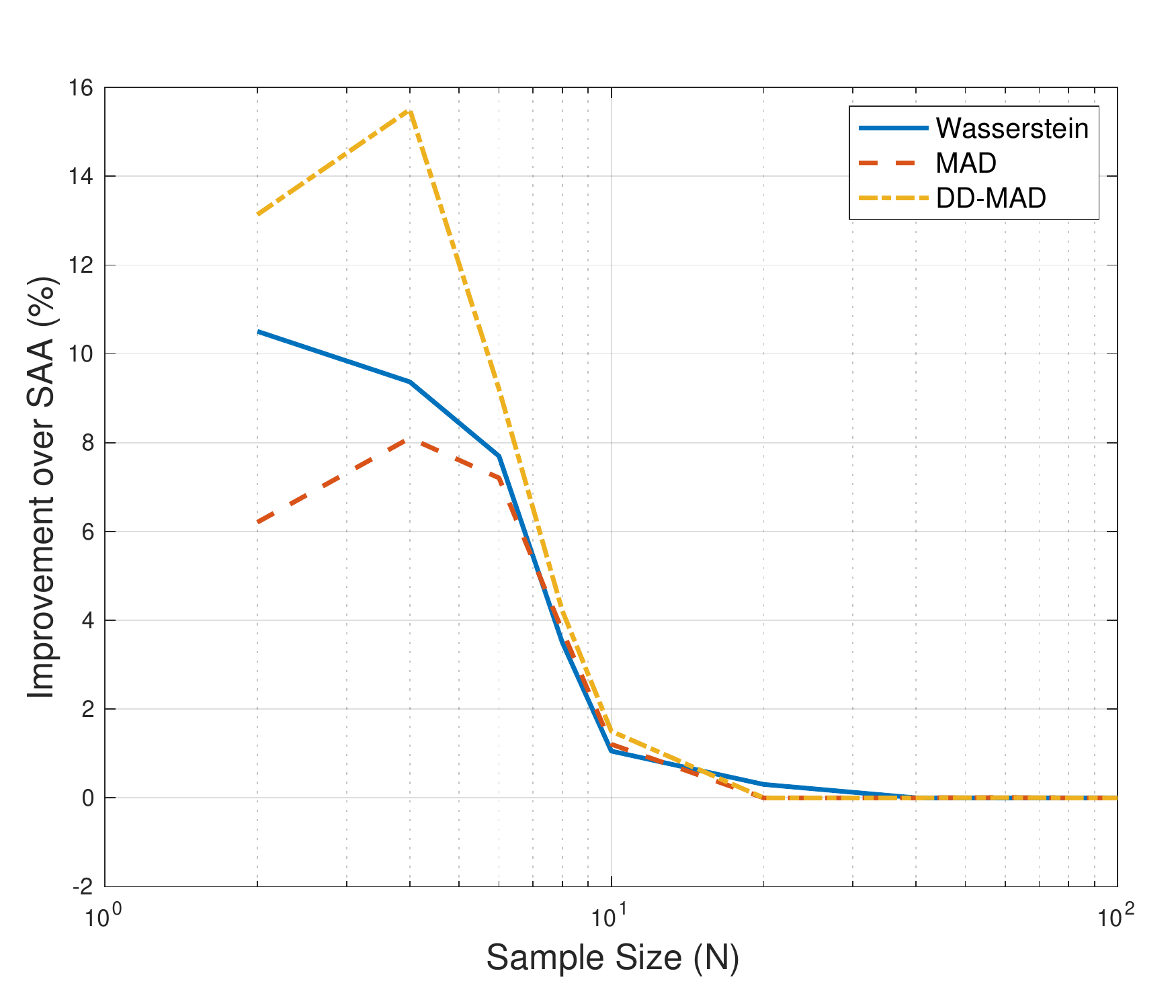}\label{fig:so_oss_impro}}
	\caption{Improvements of the DRO policies relative to the SAA policy in terms of the social optimizer and revenue maximizer respectively.}	
	\label{fig:improvement_so}
\end{figure}

Figure \ref{fig:improvement_so} depicts the out-of-sample performances of a social optimizer and a revenue optimizer under different DRO policies with $R=10$, $C=1$ and $\mu=1$. The expected values and $95$ percentiles are computed from $20$ independent trials. The $y$-axis represents the improvements of the DRO policies relative to the SAA policy, while the $x$-axis denotes the sample size. In the social optimization problem, the curve of the Wasserstein model terminates at $N=6$ since the solver fails to converge when the sample size reaches $8$. 
Meanwhile, we observe that the Wasserstein model dominates the SAA model uniformly across all sample sizes in the revenue maximization problem, while the MAD and DD-MAD models outperform the SAA model in moderate sample sizes. This is because the Wasserstein ambiguity set converges to the true distribution as the number of samples grows, whereas the moment ambiguity sets fails to converge to the true distribution.
We also find that the MAD model performs poorly when the sample size is small because the empirical MAD constitutes a biased estimator with significant estimation errors. On the other hand, the DD-MAD model---by optimizing in view of the most adverse mean and MAD---mitigates the detrimental effects of poor empirical estimations and generates high-quality policies. Finally, we observe that the advantages of the DRO policies relative to the SAA method are generally more substantial for the $95$th percentiles. This underlines a major advantage of incorporating the DRO scheme, as it reduces the likelihood of realizing extremely poor performance in the out-of-sample test.


\begin{table}[]
    \centering
    \begin{tabular}{|l|p{2.5cm}|r|r|r|r|r|r|}
         \multicolumn{2}{c}{}& \multicolumn{6}{c}{Sample size $N$ } \\
         \hline
         &Model Name&\multicolumn{1}{r}{2}&\multicolumn{1}{r}{5}&\multicolumn{1}{r}{10}&\multicolumn{1}{r}{25}&\multicolumn{1}{r}{50}&100\\
         \hline
         &MAD& \multicolumn{1}{r}{24.72}&\multicolumn{1}{r}{21.38}&\multicolumn{1}{r}{31.42}&\multicolumn{1}{r}{26.65}&\multicolumn{1}{r}{24.30}&29.84\\
         Social&DD-MAD&\multicolumn{1}{r}{33.58}&\multicolumn{1}{r}{27.49}&\multicolumn{1}{r}{22.75}&\multicolumn{1}{r}{32.94}&\multicolumn{1}{r}{27.61}&28.71\\
         &Wasserstein&\multicolumn{1}{r}{38.19}&\multicolumn{1}{r}{88.36}&\multicolumn{1}{r}{-}&\multicolumn{1}{r}{-}&\multicolumn{1}{r}{-}&-\\
         \hline
          &MAD& \multicolumn{1}{r}{0.05}&\multicolumn{1}{r}{0.03}&\multicolumn{1}{r}{0.04}&\multicolumn{1}{r}{0.05}&\multicolumn{1}{r}{0.07}&0.06\\
         Revenue&DD-MAD&\multicolumn{1}{r}{1.54}&\multicolumn{1}{r}{1.79}&\multicolumn{1}{r}{1.42}&\multicolumn{1}{r}{1.81}&\multicolumn{1}{r}{1.65}&1.59\\
         &Wasserstein&\multicolumn{1}{r}{1.69}&\multicolumn{1}{r}{1.92}&\multicolumn{1}{r}{2.41}&\multicolumn{1}{r}{2.63}&\multicolumn{1}{r}{2.95}&4.68\\
         \hline
    \end{tabular}
    \caption{Running time (in seconds) of different methods. The '-' symbol indicates that the model fails to converge in the maximal iteration/time.}
    \label{tab:time}
\end{table}

Table \ref{tab:time} reports the computation time of different models with the sample size varying from $2$ to $100$. We set the length of the confidence intervals and the radius of the Wasserstein ball to $0.1$. In this  experiment, the running time limit of SDPT3 is set to $600$ seconds and the number of iterations is set to $5000$. All computational times are averaged over $10$ trials.

The results in Table \ref{tab:time} indicate that the computational times of the MAD and DD-MAD models are size-invariant in the social optimization problem because the number of constraints is independent of the number of samples. The Wasserstein model is applicable to small-size problems. However, it encounters computational difficulties for moderate-size problem instances: when the sample size reaches $10$, the model diverges or fails to converge within the time/iteration limit. The MAD model is extremely efficient for the revenue maximization problem as it admits a closed-form solution. The DD-MAD model is still size-invariant, and its linear programming reformulation yields a much shorter computational time than the SDP reformulation in the social optimization problem. In addition, the Wasserstein model can be solved efficiently in the revenue maximization problem even for large sample sizes, benefiting from the linear programming reformulation.

In summary, the experimental results highlight the substantial advantage of employing the proposed DD-MAD distributionally robust model, particularly when limited number of observations is available to infer the underlying data-generating distribution. It yields attractive out-of-sample performances while can be solved very efficiently.

\section{Conclusion}
This paper developed an extension of Naor's strategic queue model with uncertain arrival rates using the DRO framework. We showed that under the DRO setting, the optimal threshold of an individual optimizer coincides with Naor's original result, and there exist optimal thresholds of the social and revenue optimizers not larger than the optimal individual threshold. We then proved that the revenue rate function is concave, while the social benefit rate function is concave or unimodal under some mild conditions. These nice properties lead to a closed-form solution for the revenue maximization problem and an analytical solution for the social optimization problem. 

Next, we considered the data-driven optimization setting, where decision makers only have access to limited historical samples. We proposed a data-driven MAD model by introducing an extra layer of robustness to the primitive MAD ambiguity set. As the model mitigates the detrimental estimation errors from the empirical mean and MAD, it achieves attractive performance in out-of-sample tests. We derived an SDP reformulation for the social optimization problem and a linear programming reformulation for the revenue maximization problem. We further established finite-sample guarantees for the data-driven model, which provide valuable guidance for choosing the robustness parameters in practice. Our experimental results show that a system manager who disregards ambiguities in the distribution on the arrival rate, as well as errors from the empirical parameter estimations, may incur large out-of-sample costs.
Future work includes extending the DRO scheme to the unobservable strategic queues, where newly arrived customers cannot observe the current length of the queue system.


\paragraph{Acknowledgements}
This research was supported by the National Science Foundation grant no.~$1752125$. 
	
\bibliographystyle{abbrv}
\bibliography{myreference.bib}    

\appendix \section{Proofs of Section \ref{sec:DROM}} \label{sec:proof_sec3}
\begin{lem}
The first and second derivatives of the social benefit rate function $f_n(\rho)$ are continuous.
\end{lem}
\begin{proof}
To show the continuity of the first and second derivative of $f_n(\rho)$, we will show that
\begin{equation}\label{eq:f_n(rho)_cont}
    f_n(\rho)=R\mu\left(1-\frac{1}{\sum_{k=0}^{n} \rho^k} \right)-C\left(\frac{\rho(\sum_{k=0}^{n-1}\rho^k)+\rho^2(\sum_{k=0}^{n-2}\rho^k)+\ldots+\rho^n}{(\sum_{k=0}^{n}\rho^k)} \right),
\end{equation}
which has continuous first and second derivatives.

First, we perform the transformation for the term $\frac{\rho (1-\rho^n)}{1-\rho^{n+1}}$ when $\rho \neq 1$. Note that $\frac{\rho (1-\rho^n)}{1-\rho^{n+1}}=1-\frac{1-\rho}{1-\rho^{n+1}}$, and the denominator is equal to $(1-\rho)(1+\rho+\rho^2+\ldots+\rho^{n})$. We can consequently rewrite the first term as
\begin{equation*}
    \frac{\rho (1-\rho^n)}{1-\rho^{n+1}}=1-\frac{1}{\sum_{k=0}^{n} \rho^k}.
\end{equation*}

Next, we prove the equivalence of the remaining part $\frac{ (n+1)\rho^{n+1}}{1-\rho^{n+1}}-\frac{ \rho}{1-\rho}$ when $\rho \neq 1$. Similarly, by the fact that $(1-\rho^{n+1})=(1-\rho)(\sum_{k=0}^{n} \rho^k)$, we can rewrite this part as
\begin{align*}
    \frac{ (n+1)\rho^{n+1}}{1-\rho^{n+1}}-\frac{ \rho}{1-\rho}&=\frac{(n+1)\rho^{n+1}}{(1-\rho)(\sum_{k=0}^{n}\rho^k)}-\frac{\rho(\sum_{k=0}^{n}\rho^k)}{(1-\rho)(\sum_{k=0}^{n}\rho^k)}\\
    &=\frac{-\rho-\rho^2-\ldots-\rho^{n}+n\rho^{n+1}}{(1-\rho)(\sum_{k=0}^{n}\rho^k)}\\
    &=\frac{\rho^{n+1}-\rho+\rho^{n+1}-\rho^2+\ldots+\rho^{n+1}-\rho^n}{(1-\rho)(\sum_{k=0}^{n}\rho^k)}\\
    &=\frac{\rho(\rho-1)(1+\rho+\ldots+\rho^{n-1})+\rho^2(\rho-1)(1+\rho+\ldots+\rho^{n-2})+\ldots+\rho^n(\rho-1)}{(1-\rho)(\sum_{k=0}^{n}\rho^k)}\\
    &=-\frac{\rho(\sum_{k=0}^{n-1}\rho^k)+\rho^2(\sum_{k=0}^{n-2}\rho^k)+\ldots+\rho^n}{(\sum_{k=0}^{n}\rho^k)}.
\end{align*}

When $\rho=1$, $R\mu \left(1-\frac{1}{\sum_{k=0}^{n} \rho^k} \right) - C\left(\frac{\rho(\sum_{k=0}^{n-1}\rho^k)+\rho^2(\sum_{k=0}^{n-2}\rho^k)+\ldots+\rho^n}{(\sum_{k=0}^{n}\rho^k)} \right)=R\mu\left(1-\frac{1}{1+n}\right) - C \frac{n}{2}$, which coincides with $f_n(1)$. Therefore, $f_n(\rho)$ is equal to \eqref{eq:f_n(rho)_cont}. One can verify that the first and second derivatives of \eqref{eq:f_n(rho)_cont} are continuous; hence, $f_n(\rho)$ also has these properties.
\end{proof}

\begin{lem}\label{lem:rm_concave2}
    The function $h_n(\rho) = \frac{\rho(1-\rho^n)}{1-\rho^{n+1}}$ is strictly concave and monotone increasing on $[0,1) \cup (1, \infty)$.
\end{lem}

\begin{proof}
    
    
    When $\rho \in [0,1) \cup (1, \infty)$, the first derivative of $h_n(\rho)$ is $$h_n'(\rho)= \frac{n \rho^{n+1} - (n+1) \rho^n + 1}{(1-\rho^{n+1})^2}.$$ Define the numerator as $\varphi_n(\rho) = n \rho^{n+1} - (n+1) \rho^n + 1$. The first derivative of $\varphi_n(\rho)$ is given by $\varphi_n'(\rho) = n(n+1)\rho^{n-1}(\rho-1).$ Note that when $0 < \rho < 1$, $\varphi_n'(\rho)$ is negative, and when $\rho > 1$, $\varphi_n'(\rho) $ is positive. Therefore, the function $\varphi_n(\rho)$ is decreasing on $(0,1)$ and increasing on $(1, \infty)$. Meanwhile, by the fact that $\varphi_n(1)= 1 + n - (n+1) = 0$, we know that the numerator $\varphi_n(\rho)$ is positive on $[0,1) \cup (1, \infty)$. Since the denominator  $(1-\rho^{n+1})^2$  is  positive, the first derivative $h
'_n(\rho)$  is positive on $[0,1) \cup (1,\infty)$. Thus, we conclude that $h_n(\rho)$ is increasing on $[0,1) \cup (1, \infty)$.
    
    Next, we show the second derivative of $h_n(\rho)$ is negative. We have
    $$h_n''(\rho)= \frac{(n+1)\rho^{n-1} [n\rho^{n+2}-(n+2) \rho^{n+1} +(n+2)\rho -n]}{(1-\rho^{n+1})^3}.$$ Since the term $\frac{(n+1)\rho^{n-1}}{(1-\rho^{n+1})^3}$ is positive on $[0,1)$ and is negative on $(1,\infty)$, we simply need to determine the sign of $ [n\rho^{n+2}-(n+2) \rho^{n+1} +(n+2)\rho -n]$. For convenience, define 
    \[\psi_n(\rho)\coloneqq n\rho^{n+2}-(n+2) \rho^{n+1} +(n+2)\rho -n.\]
    Note that $\psi_n (0) = -n <0$ and $\psi_n (1) = 0$, while $\lim_{\rho \rightarrow \infty} \psi_n (\rho) = +\infty$. Therefore, if $\psi_n(\rho)$ is increasing on $[0,1) \cup (1, \infty)$, the second derivative $h''_n(\rho)$ will be negative on $[0,1) \cup (1, \infty)$. To show this, we take the first derivative of $\psi_n(\rho)$ and obtain $$ \psi_n'(\rho) = n(n+2)\rho^{n+1} - (n+2)(n+1)\rho^n+(n+2).$$ 
    Taking specific values into this function we can obtain $\psi_n'(0)= n+2 >0$, $\psi_n'(1) = 0$ and $\lim_{\rho \rightarrow \infty} \psi_n(\rho) = +\infty$. Similarly, if $\psi_n'(\rho)$ is decreasing on $[0,1)$ and increasing on $(1,\infty)$, then $\psi'_n(\rho)$ will be positive on $[0,1) \cup (1, \infty)$. To verify this, we can take the second derivative of $\psi_n(\rho)$, which gives $$\psi_n''(\rho) = (n+2)(n+1)n\rho^{n-1}(\rho - 1). $$ One can verify that
    $\psi_n''(\rho)$ is negative on $[0,1)$ and positive on $(1,\infty)$. Thus, we have established that $h''_n(\rho)$ is negative on $[0,1) \cup (1, \infty)$ and  $h_n(\rho)$ is concave on $[0,1) \cup (1, \infty)$.
\end{proof}

\begin{lem}\label{lem:rpart_concave}
    For any $v \in\RR$, $v\geq 1$, the function $g_v(\rho) = \frac{v\rho}{1-\rho^{v 
    }}-\frac{\rho}{1-\rho}$ is concave on $[0,1)$.
\end{lem}

\begin{proof}
    For any $v \in\RR$, $v\geq 1$, one can verify that $g_v(\rho)$ is continuous and second order differentiable on $[0,1)$. Thus, $g_v(\rho)$ is concave if and only if its second derivative $$ g_v''(\rho) = \frac{v^2 \rho^{v-1} (1 + v + \rho^v(v-1))}{(1 - \rho^v)^3} - \frac{2}{(1-\rho)^3}$$ 
    is non-positive for every $\rho\in[0,1)$. Notice that when $\rho=0$, $g_v''(\rho)=-2$ is less than zero. We now prove that the second derivative is also non-positive on $(0,1)$. We first observe that $g_v''(\rho)=0$ at $v=1$ for all $\rho\in[0,1)$. Consider the partial derivative with respect to $v$:
    $$\partial_v g_v''(\rho)={\frac {v \left(  \left(  \left( {v}^{2}-v \right) \ln  \left( \rho  \right) -3\,v+2 \right) {\rho}^{3\,v-1}+ \left( 4\,{v}^{2}\ln  \left( \rho  \right) -4 \right) {\rho}^{2\,v-1}+{\rho}^{v-1} \left(  \left( {v}^{2}+v  \right) \ln  \left( \rho \right) +3\,v+2 \right)  \right) }{ \left( 1- {\rho}^{v} \right) ^{4}}}.$$
    If this function is non-positive for all $v\in\RR$, $v\geq 1$,  then we can establish that the second derivative $g_v''(\rho)$ is non-positive for all $\rho\in(0,1)$. 
    
    Consider a fixed $v\in\RR$, $v\geq 1$. 
    Defining $\psi(\rho)$ as the product of $\partial_v g_v''(\rho)$ and $\frac{(1-\rho^v)^4}{v \, \rho^{v-1}}>0$ yields
    \begin{equation}\nonumber
        \psi_v(\rho)\coloneqq\left( v \left( v-1 \right) {\rho}^{2\,v}+4\,{v}^{2}{\rho}^{v}+{v}^{2}+v \right) \ln  \left( \rho \right) + \left( -3\,v+2 \right) {\rho}^{2\,v}-4 \,{\rho}^{v}+3\,v+2.
    \end{equation}
    We  show $\psi_v(\rho)$ is non-positive for $\rho \in (0,1)$. Observe that $\psi_n(\rho)$ goes to negative infinity as $\rho \rightarrow 0_+$ and equals to zero at $\rho=1$. Thus, it is sufficient to show that $\psi_v(\rho)$ is increasing on $\rho \in (0,1)$ for every fixed $v$. Taking the derivative with respect to $\rho$ and dividing it by $v\rho^{v-1} > 0$ yields
    \begin{equation}\nonumber
        \tau_n(\rho) \coloneqq \frac{\psi'_n(\rho)}{v\rho^{v-1}}=2\,v \left(  \left( v-1 \right) {\rho}^{v}+2\,v \right) \ln  \left( \rho \right) + \left( v+1 \right) {\rho}^{-v}+ \left( -5\,v+3 \right) {\rho}^{v} +4\,v-4 .
    \end{equation}
    Similarly, one can verify that this expression goes to positive infinity as $\rho \rightarrow 0_+$ and is equal to zero at $\rho=1$. Therefore, to show that ${\psi'_n(\rho)}$ is positive on $(0,1)$, it is sufficient to show that $\tau_n(\rho)$ is decreasing on $(0,1)$. Again, taking the derivative with respect to $\rho$ and dividing it by $v\rho^{v-1}>0$, we get 
    \begin{equation*}
        \varphi(\rho)\coloneqq \frac{\tau'_n(\rho)}{v\rho^{v-1}}=2v\left( v-1 \right) \ln  \left( \rho \right) - \left( v+1 \right) {\rho}^{-2\,v}+4v\,{\rho}^{-v} - 3\,v+1.
    \end{equation*}
    This expression again vanishes at $\rho=1$ and goes to negative infinity as $\rho \rightarrow 0_+$. Thus, it is sufficient to show it is increasing on $(0,1)$. Taking the derivative with respect to $\rho$ and multiplying with $\frac{\rho^{2v+1}}{2v}>0$ yield:
    \begin{equation*}
        \theta_n(\rho)\coloneqq\frac{\varphi'(\rho)\rho^{2v+1}}{2v}=\left( v-1 \right) {\rho}^{2\,v}-2\,v\,{\rho}^{v}+v+1.
    \end{equation*}
     At $\rho=0$, $\theta_n(\rho)$ is equal to $v+1$, which is greater than zero, and vanishes  at $\rho=1$. 
     Taking the derivative with respect to $\rho$ and dividing by $2n\, \rho^{n-1}>0$, we have 
     $$\phi_n(\rho)\coloneqq \frac{\theta_n'(\rho)}{2n\, \rho^{n-1}}=(v-1)\rho^v - v.$$ 
     One can verify that when $v\geq 1$, $\phi(\rho)$ is always non-positive, which completes our proof.
\end{proof}

\begin{proof}[Proof of Lemma \ref{lem:so_propertie} statement (1)]
    Using the lemmas above, we are ready to show that when $\frac{R\mu}{C} \geq n+1$, the social benefit rate function $f_n(\rho)$ is strictly concave on $[0,1]$.
    For $\rho \in [0,1)$, we can rewrite~$f_n(\rho)$ as $$f_n(\rho) = \left(R \mu - C(n+1)\right) \frac{\rho (1-\rho^n)}{1-\rho^{n+1}} + C \frac{(n+1)\rho}{1-\rho^{n+1}} - \frac{C \rho}{1-\rho}.$$ 
    
    From Lemma \ref{lem:rm_concave2} and Lemma \ref{lem:rpart_concave}, we know that  $\frac{\rho (1-\rho^n)}{1-\rho^{n+1}}$ is strictly concave and $ \frac{(n+1)\rho}{1-\rho^{n+1}} - \frac{ \rho}{1-\rho}$ is concave. Therefore, $f_n(\rho)$ is the sum of a strictly concave and a concave function, which is  strictly concave for $\rho \in [0,1)$.  
\end{proof}


\begin{proof}[Proof of Lemma \ref{lem:so_propertie} statement (2)]
    When $n=1$, one can  verify that $f_1(\rho)$ is a concave increasing function for $\rho \in \RR_+$. We now proceed to show the function is unimodal for $n \geq 2$. A sufficient condition for $f_n(\rho)$ to be unimodal is $f_n'(0) > 0$, $\lim_{\rho \rightarrow \infty} f_n'(\rho)<0$, and $f_n'(\rho)=0$ has a unique solution. Taking the derivative of $f_n(\rho)$ yields:
	\begin{equation*}
	f_n'(\rho)=\left\{\begin{array}{ll}
	     R \mu \left(\frac{n\rho^{n+1} -(n+1) \rho^{n} +1 }{(\rho^{n+1}-1)^2} \right) + C \left(\frac{(n+1) \left(n\rho^{n+1} + 1 \right)}{(\rho^{n+1}-1)^2} - \frac{1}{(\rho-1)^2}\right) \quad &\rm{if} \ \rho \neq 1\\
	      \lim_{\rho \rightarrow 1}R \mu \left(\frac{n\rho^{n+1} -(n+1) \rho^{n} +1 }{(\rho^{n+1}-1)^2} \right) + C \left(\frac{(n+1) \left(n\rho^{n+1} + 1 \right)}{(\rho^{n+1}-1)^2} - \frac{1}{(\rho-1)^2}\right) & \rm{if} \ \rho=1. \\
	\end{array}\right.
	\end{equation*}
    
    Showing $f_n'(\rho)=0$ has exactly one positive root directly is non-trival. However, it is equivalent to showing $(1-\rho)^2 f_n'(\rho)=0$ has exactly three positive roots. One can verify that this new term can be written explicitly as $(1-\rho)^2 f_n'(\rho) = R \mu \left(\frac{n\rho^{n+1} -(n+1) \rho^{n} +1 }{(1+\rho+\ldots+\rho^n)^2} \right) + C \left(\frac{(n+1) \left(n\rho^{n+1} + 1 \right)}{(1+\rho+\ldots+\rho^n)^2} - 1\right),\ \forall \rho \in \RR_+$. We then reformulate the root equation to a polynomial form:
    \begin{align}
		\nonumber   &R \mu \left(\frac{n\rho^{n+1} -(n+1) \rho^{n} +1 }{(1+\rho+\ldots+\rho^n)^2} \right) + C \left(\frac{(n+1) \left(n\rho^{n+1} + 1 \right)}{(1+\rho+\ldots+\rho^n)^2} - 1\right)=0\\
		\nonumber  &\Longleftrightarrow \\
		\nonumber &R\mu (n\rho^{n+1} -(n+1) \rho^{n} +1) + C(n+1)(n\rho^{n+1}+1) = C(1+\rho + \ldots + \rho^n)^2 \\
		\nonumber  &\Longleftrightarrow\\
		\nonumber& C(1+\rho + \ldots + \rho^n)^2 - n(R\mu + C(n+1))\rho^{n+1} +R\mu (n+1)\rho^n -R\mu-C(n+1)=0.
	\end{align}
	The left-hand side of the equation is a single variable polynomial, and one can verify that it has three sign changes. Based on Descartes' rule of signs, the number of positive roots is at most three. By the fact that $ f_n'(0) > 0$ and $\lim_{\rho \rightarrow \infty} f_n'(\rho) < 0$, $f_n'(\rho)$ must has at least one root. Since the term $(1-\rho)^2$ has two roots, we know this polynomial has at least three roots. Therefore, this polynomial has exactly three roots and $f_n'(\rho)$ has exactly one root. This shows that $f_n(\rho)$ is a unimodal function.
\end{proof}

\begin{proof} [Proof of Lemma \ref{lem:so_propertie} statement (3)]      
    The second derivative of $f_n(\rho)$ is 
    \begin{equation*}
        f_n''(\rho)=\left\{\begin{array}{lll}
         R\mu\frac{(n+1)\rho^{n-1}((n+1)(\rho-1)(\rho^{n+1}+1)-2\rho(\rho^{n+1}-1))}{(1 - \rho^{n+1})^3}&\\
         \qquad +C\left(\frac{(n+1)^2 \rho^{n} (2 + n + n\rho^{n+1})}{(1 - \rho^{n+1})^3} - \frac{2}{(1-\rho)^3}\right) &\textup{if} \ \rho \neq 1\\
          \lim_{\rho \rightarrow 1}R\mu\frac{(n+1)\rho^{n-1}((n+1)(\rho-1)(\rho^{n+1}+1)-2\rho(\rho^{n+1}-1))}{(1 - \rho^{n+1})^3} & \\
          \qquad +C\left(\frac{(n+1)^2 \rho^{n} (2 + n + n\rho^{n+1})}{(1 - \rho^{n+1})^3} - \frac{2}{(1-\rho)^3}\right)&\textup{if} \ \rho = 1.
    \end{array}\right.
	\end{equation*}
	
	Showing $f_n''(\rho)=0$ only has one root is equivalent to showing $(1-\rho)^3f_n(\rho)=0$ has exactly four roots. Once can check that $(1-\rho)^3f_n''(\rho)$ coincides with $R\mu\frac{(n+1)\rho^{n-1}((n+1)(\rho-1)(\rho^{n+1}+1)-(\rho+1)(\rho^{n+1}-1))}{(1+\rho + \ldots +\rho^n)^3} +C\left(\frac{(n+1)^2 \rho^{n} (2 + n + n\rho^{n+1})}{(1+\rho + \ldots +\rho^n)^3} - 2\right).$ Similar to the previous proof, we transform the root equation to a polynomial form:
	\begin{align}
        \nonumber &(1-\rho)^3 f_n''(\rho)=0\\
		\nonumber  &\Longleftrightarrow \\
		\nonumber &R\mu\frac{(n+1)\rho^{n-1}((n+1)(\rho-1)(\rho^{n+1}+1)-(\rho+1)(\rho^{n+1}-1))}{(1+\rho + \ldots +\rho^n)^3} +C\frac{(n+1)^2 \rho^{n} (2 + n + n\rho^{n+1})}{(1+\rho + \ldots +\rho^n)^3} \\
		\nonumber &\qquad \qquad  - 2C=0\\
		\nonumber  &\Longleftrightarrow \\
		\nonumber &2C(1+\rho+\ldots+\rho^n)^3+(n+1)\rho^{n-1}(R\mu(n+1)-\left(4C(n+1)+R\mu(n+3)\right)\rho-R\mu(n+1)\rho^2 \\
		\nonumber & + R\mu (n+1)\rho^{n+1}+ (R\mu(n-1)+C(n^2+n)\rho^{n+2})=0.
	\end{align}
	One can verify that this polynomial has four sign changes. Based on Descartes' rule of signs, the number of positive roots is four or two. Since the term $(1-\rho)^3$ already has three roots, $f_n''(\rho)$ has exactly one root, which also implies the sign of $f_n''(\rho)$ changes at most once.
\end{proof}

    \begin{proof}[Proof of Lemma \ref{lem:so_3points}]
    
    We first show that strong duality holds and both the primal and dual optimal solutions are attained, which is a sufficient condition for complementary slackness. To show this, we need to prove both the primal and dual problems have interior points. 
    
    Showing the existence of interior points of the primal problem is equivalent to finding a point $(1,m,d)$ that resides in the interior of the convex cone
    \begin{equation}\nonumber
        \mathcal V =\left\{(l,t,u) \in \RR^3: \ \exists \nu \in \mathcal M_+ \ \rm{such \ that} 
        \hspace{-3mm}\begin{array}{ll}
             &\int_{\Xi} \nu(\rm d\rho)= \mathit l  \\
             &\int_{\Xi} \rho \, \nu(\rm d\rho)=\mathit{t}\\
             &\int_{\Xi} |\rho-t| \, \nu(\rm{d}\rho)=\mathit u
        \end{array}
        \right\},
    \end{equation}
	where $\Xi=[a,b]$. We define $\mathbb B_\kappa(c)$ by the closed Euclidean ball of radius $\kappa \geq 0$ centered at $c$. To this end, choose any point $(l_s,t_s,u_s) \in \mathbb B_\kappa(1) \times \mathbb B_\kappa(m) \times \mathbb B_\kappa(d)$ with sufficiently small $\kappa>0$, and consider the  measure
	\begin{equation}\nonumber
	    \nu_s=\frac{n_s}{2(m_s-a)} \cdot \delta_a+\left(l_s-\frac{u_s}{2(t_s-a)}-\frac{u_s}{2(b-t_s)}\right)\cdot \delta_t + \frac{u_s}{2(b-t_s)}\cdot \delta_b,
	\end{equation}
    where $s \cdot \delta_m$ denotes a measure that places mass $s$ at $m$. By construction, this measure satisfies $\int_{\rho} \nu_s(\rm d\rho)= \mathit l_s$, $\int_{\rho} \rho \, \nu_s(\rm d\rho)=\mathit{t_s}$ and $\int_{\rho} |\rho-t_s| \, \nu_s(\rm{d}\rho)=\mathit u_s$ for a sufficiently small $\kappa$ (since $m \in (a,b)$ and $d \in (0,\frac{2(m-a)(b-m)}{b-a})) $. 
    Therefore, strong duality holds and the optimal values of the primal problem and the dual problems coincide. Moreover, as there exist interior points of the primal problem and the common optimal value is finite, we have the dual optimal solution is also attained~\cite[Proposition 3.4]{shapiro2001duality}.
    Noticing that the support $[a,b]$ is compact, while the social benefit rate function $f_n(\rho)$ and the moment functions $\rho$ and $|\rho-m|$ are continuous, we can invoke  \cite[Corollary 3.1]{shapiro2001duality} to establish that the primal optimal solution is attained.
     To this end, we have strong duality holds and both the primal and  dual optimal solutions are attained, which implies complementary slackness holds~\cite[Proposition 2.1]{shapiro2001duality}.
    
    \end{proof}

\begin{proof}[Proof of Lemma \ref{lem:rm_concave}]
    We know that the revenue rate function is continuous for $\rho \in \mathbb R_+$. Therefore, employing Lemma \ref{lem:rm_concave2} completes the proof.
\end{proof}

	    \begin{proof}[Proof of Proposition \ref{prop:so_mad}]
	The dual problem \eqref{eq:so_mean_abs_dual} can be equivalently written as
    \begin{equation*}
	\begin{array}{ccll}
	&\displaystyle \sup_{\alpha,\beta,\gamma \in \RR}& \displaystyle \mathbb E_{\mathbb P} \left[ \alpha|\rho-m| +\beta \rho + \gamma \right]\\
	&\st&  \alpha|\rho-m| +\beta \rho + \gamma \leq f_n(\rho) \quad \quad \forall \rho \in [a,b],
	\end{array} 
	\end{equation*}
	where $\mathbb P\in\mathcal P$ is an arbitrary probability measure in the ambiguity set. Observe that the left-hand side of the constraint is a two-piece piecewise affine function with a breakpoint at the mean $m$. Therefore, we can interpret the dual problem as finding a feasible two-piece piecewise affine function with the largest expected value. We now use this interpretation to derive the desired results.
	
	First, we illustrate the case when $f_n(b) + f'_n(b) (m-b) \geq f_n(m)$. The constraint of the dual problem indicates that $f_n(\rho)$ majorizes $\alpha|\rho-m| +\beta \rho + \gamma$. One can verify that the two-piece piecewise affine function with the largest expected value is the one that touches $f_n(\rho)$ at three points: $\rho=a,m$ and $b$; see Figure \ref{fig:mad_impro_1} for an illustrative example.  By complementary slackness in Lemma \ref{lem:so_3points}, the optimal distribution can only assign positive mass to these three points, which yields the following system of linear equations: 
    \begin{equation*}
   \label{eq:so_mad_3point_prob1}
   \begin{array}{cc}
        &  p_1(a-m) + p_2(m-m) + p_3(b-m) = m\\
        &  p_1|a-m| + p_2|m-m| + p_3|b-m|= d\\
        &  p_1 + p_2 + p_3 = 1.
   \end{array}
   \end{equation*}
   Solving this system of linear equations leads to the first result in Proposition \ref{prop:so_mad}.
   
   Next, we prove the two cases when $f_n(b) + f'_n(b) (m-b) \leq f_n(m)$. 
    If $0<d<  d_0\coloneqq\frac{2(m-a)(\rho_{t}-m)}{\rho_t-a}$, we claim that the extremal distribution that solves \eqref{eq:so_mean_abs_primal} is a three-point distribution. To see this, we know that complementary slackness holds from Lemma \ref{lem:so_3points}, which means the extremal distribution is supported on points where the dual constraint is binding. Since the two-piece piecewise affine function can touch $f_n(\rho)$ on at most three points under constraint \[\alpha|\rho-m| +\beta \rho + \gamma \leq f_n(\rho) \quad \quad \forall \rho \in [a,b],
    \] 
    the extremal distribution is either a one-point, two-point, or a three-point distribution. We readily exclude the possibility that the extremal distribution is a one-point distribution because the mean-absolute deviation of a one-point distribution is zero. Next,
    we illustrate why the extremal distribution cannot be a two-point distribution. Suppose there exists a two-point distribution supported on $\{\rho_1,\rho_2\}$ that solves the worst-case expecation problem. Then, by complementary slackness, the dual constraint $
   f_n(\rho)=\alpha|\rho-m| +\beta \rho + \gamma$ will be binding at these two points. Without loss of generality, we assume $\rho_1 \in [a,m)$ and $\rho_2 \in (m,b]$. Since $f_n(\rho)$ is strictly concave for $\rho \in [a,m)$ and the dual constraint requires $ \alpha|\rho-t| +\beta \rho + \gamma \leq f_n(\rho)$, we thus have $\rho_1=a$. Since $\rho_t$ is defined as the $\rho$ coordinate of the point such that the line segment between $(m,f_n(m))$ and $(\rho_t,f_n(\rho_t))$ is tangent with $f_n(\rho)$, we must have $\rho_2 \geq \rho_t$; otherwise, the dual constraint will be violated. Since $\rho_2-\rho_1 \geq \rho_t - a$, the corresponding mean-absolute deviation will be greater than $d_0$. Therefore, the extremal distribution cannot be a two-point distribution, i.e., it is a three-point distribution. 
    Next, it can be shown that if $f_n(\rho)$ intersects $\alpha|\rho-m| +\beta \rho + \gamma$ at three points, then these three points must be  $\rho=a,m$ and $\rho_t$. Therefore, we have the following system of linear equations:
    \begin{equation*}  
   \label{eq:so_mad_3point_prob}
   \begin{array}{cc}
        &  p_1(a-m) + p_2(m-m) + p_3(\rho_t-m) = m\\
        &  p_1|a-m| + p_2|m-m| + p_3|\rho_t-m|= d\\
        &  p_1 + p_2 + p_3 = 1.
   \end{array}
   \end{equation*}
    Solving this system of linear equations leads to the second result in Proposition \ref{prop:so_mad}.
   	
   	We now establish that if $d_0\leq d$, the extremal distribution is a two-point distribution. Similarly, by the fact that the extremal distribution is a discrete distribution supported on at most three points, we just need to show there does not exist a one-point or three-point extremal distribution that solves \eqref{eq:so_mean_abs_primal}. We can exclude the possibility of one-point distribution easily, since its mean-absolute deviation is 0. As we described previously, the extremal three-point distribution is supported on $\rho=a,\rho_t$ and $b$, and the largest mean-absolute deviation that can be achieved within this support is given by $ d_0 \coloneqq\frac{2(m-a)(\rho_{t}-m)}{\rho_t-a}$. Since $d\geq d_0$, the extremal distribution can only be a two-point distribution. One of the support points is given by $\rho=a$, while the other one is determined by the value of $d$, which yields the following linear equations: 
   	
   \begin{equation}
   \label{eq:so_mad_2point_prob}
   \begin{array}{cc}
        &  p_1(a-m) + p_2(\rho_2-m) = m\\
        &  p_1|a-m| + p_2|\rho_2-m|= d\\
        &  p_1 + p_2 = 1.
   \end{array}
   \end{equation}
   Solving this system of equations, we obtain the optimal solution explicitly as:
   $$ p_1 = \frac{d}{2(m-a)},\ \rho_1=a; \ p_2 = 1-\frac{d}{2(m-a)},\ \rho_2 = \frac{da+2m(a-m)}{d+2(a-m)}.$$ 
   This completes the proof. 
    \end{proof}

\section{Proofs of Section \ref{sec:data-driven-mad}}\label{sec:proof_sec4}
	\begin{proof}[Proof of Theorem~\ref{thm:data-driven-social}]
	Problem \eqref{eq:ddmad_inf} can be equivalently written as:
	\begin{equation*}
	\begin{array}{ccll}
	\vspace{1mm}&\displaystyle \inf_{\nu \in \mathcal M_+}&\displaystyle \int_{\Xi} f_n(\rho) \mathbb \nu(\rm d\rho) \\
	\vspace{1mm}&\st& \displaystyle \int_{\Xi} |\rho-\hat m| \, \nu(\rm d\rho)= \textit d\\
	\vspace{1mm}&& \displaystyle \int_{\Xi} \rho \, \nu(\rm d\rho)= \textit m\\
	\vspace{1mm}&& \displaystyle \int_{\Xi} \nu(\rm d\rho)= 1\\
	\vspace{1mm}&& m_l \leq  m \leq m_u \\
	&& d_l \leq  d \leq d_u.
	\end{array} 
	\end{equation*}
	Dualizing this optimization problem yields 
		\begin{equation*}
	\begin{array}{ccll}
	&\displaystyle\sup_{\theta \in \RR^4_+,\gamma \in \RR}&\displaystyle \gamma+\theta_1 d_l - \theta_2 d_u + \theta_3 m_l - \theta_4 m_u\\
	&\st& \displaystyle (\theta_1 - \theta_2)|\rho-\hat m| +(\theta_3-\theta_4) \rho + \gamma \leq f_n(\rho) \quad \quad \forall \rho \in [a,b].
	\end{array} 
	\end{equation*}
	Applying algebraic reductions and invoking Lemma \ref{poly_lem_so1} lead to the desired reformulation. The derivation straightforwardly follows that of Theorem \ref{thm:so_mad}---we omit for brevity.
	\end{proof}

	\begin{proof}[Proof of Theorem \ref{thm:data-driven-revenue}]
	    The dual problem is given by
	\begin{equation*}
	\begin{array}{ccll}
	&\displaystyle\sup_{\theta \in \RR^4_+,\gamma \in \RR}&\displaystyle \gamma+\theta_1 d_l - \theta_2 d_u + \theta_3 m_l - \theta_4 m_u\\
	&\st& \displaystyle (\theta_1 - \theta_2)|\rho-\hat m| +(\theta_3-\theta_4) \rho + \gamma \leq r_n(\rho) \quad \quad \forall \rho \in [a,b].
	\end{array} 
	\end{equation*}
	Since the revenue rate function $r_n(\rho)$ is concave for $\rho \geq 0$, the semi-infinite constraints are satisfied if and only if each constraint is satisfied at points $\rho=a,\hat \rho_i, b$, which completes the proof.
	\end{proof}

	\section{Distributionally Robust Model with a Wasserstein Ambiguity Set}
	\label{sec:wasser}

	In this section, we study the DRO model with a Wasserstein ambiguity set \cite{gao2016distributionally,esfahani2017data}. 
	We develop solution schemes to find the optimal threshold strategies for a social optimizer and a revenue maximizer, respectively given by $\hat{n}_s$ and $\hat{n}_r$, such that the worst-case expected benefit rates are maximized.
	Here, the worst-case is taken over the Wasserstein ambiguity set containing all probability distributions (discrete or continuous)  sufficiently close to the discrete empirical distribution, where the closeness between two distributions is measured in terms of the Wasserstein metric  \cite{esfahani2018data}. 
	\begin{defn}(Wasserstein Metric)
		\label{def:wasser}
		For any $r \geq 1$, let $\mathcal M^r(\Xi)$ be the set of all probability distributions $\mathbb P$ supported on $\Xi$ satisfying $\mathbb E_{\mathbb P}[\|  \xi \|^r] = \int_\Xi \|  \xi \|^r \mathbb P(\rm d  \xi) < \infty$. The $r$-Wasserstein distance between two distributions $\mathbb P_1, \mathbb P_2 \in \mathcal P_0^r(\Xi)$ is defined as
		\begin{equation*}
		\mathcal W^r(\mathbb P_1, \mathbb P_2) = \inf\left\{\left(\int_{\Xi^2} \|  \xi_1 -  \xi_2 \|^r \mathbb Q(\rm d\xi_1, \rm d \xi_2)\right)^{\frac{1}{r}}\right\}
		\end{equation*}
		where $\mathbb Q$ is a joint distribution of $ \tilde\xi_1$ and $\tilde \xi_2$ with marginals $\mathbb P_1$ and $\mathbb P_2$, respectively.
	\end{defn} 	
	The Wasserstein distance $\mathcal W^r(\mathbb P_1, \mathbb P_2)$ can be viewed as the ($r$-th root of the) minimum cost for moving the distribution $\mathbb P_1$ to $\mathbb P_2$, where the cost of moving a unit mass from $ \xi_1$ to $ \xi_2$ amounts to $\| \xi_1 -  \xi_2 \|^r$. The
	joint distribution $\mathbb Q$ of $ \tilde\xi_1$ and $ \tilde\xi_2$  is therefore naturally interpreted as a mass transportation plan~\cite{esfahani2018data}. Similarly to the data-driven setting in Section \ref{sec:data-driven-mad}, we assume that we have observed a finite set of $N$ independent realizations 
given by $\{\hat{\rho}_i\}_{i \in [N]}$, where $\hat \rho_i = \hat \lambda_i/\mu$. Using the observations, we define the empirical distribution $\hat {\mathbb P}_N \coloneqq \frac{1}{N} \sum_{i \in [N]} \delta_{\hat \rho_i}$
as the discrete uniform distribution on the samples.

In this paper, we consider the   Wasserstein ambiguity set defined as 
	\begin{equation}
	\label{eq:wasser_ball}
 	\mathcal B_\epsilon(\mathbb {\hat P} _N) \coloneqq \left\{\mathbb P \in \mathcal P_0(\Xi): 
	  \mathcal W^1(\mathbb P, \mathbb {\hat P}_N) \leq \epsilon
	 \right\},
	\end{equation}
	    which is a neighborhood around the empirical distribution. The ambiguity set contains all distributions supported on $\Xi$  that are of type-$1$ Wasserstein distance less than or equal to $\epsilon$ from $\mathbb {\hat P}_N$. By adjusting the radius  $\epsilon$ of the ball, one can control the degree of conservatism of the DRO model. If $\epsilon = 0$, the Wasserstein ball shrinks to a singleton set containing only the empirical distribution $\mathbb{\hat P}_N$. 
	    One can further show that this data-driven DRO model converges to the corresponding true stochastic program as the sample size $N$ tends to infinity \cite{esfahani2017data}. 

 We derive the optimal threshold strategies $\hat n_s$ and $\hat n_r$ for a social optimizer and a revenue maximizer, respectively. As stated in Section \ref{sec:droqueue}, the optimal joining threshold $\tilde n_e$ for an individual customer is independent of the arrival rate, and we have $\tilde n_e = n_e$ from \eqref{eq:io}.

	\subsection{Social Optimizer}
	
	The objective of a social optimizer is to obtain an optimal joining threshold $\hat n_s$ that maximizes the worst-case expected  benefit, i.e., $\hat n_s \in \argmax_{n \in \mathbb Z_+} \{Z_s(n)\}$, where 
	\begin{equation}\label{eq:worst-case-wasser}
	Z_s(n) \coloneqq   \inf_{\mathbb P \in \mathcal B_\epsilon(\mathbb {\hat P} _N) }  \mathbb E_{\mathbb P} \left[f_n(\rhot) \right].
	\end{equation}
	The worst-case expectation is computed over all distributions in the Wasserstein ambiguity set $\mathcal B_\epsilon(\mathbb {\hat P} _N)$ with the support set $\Xi=[a,b]$. 
	\begin{thm}
		\label{wass_so}
		For any $n \geq 1$ and $\mathcal P = \mathcal B_\epsilon(\hat {\mathbb P}_N)$, 
		 the worst-case expectation $Z_s(n)$ coincides with the optimal objective value of the following semidefinite program:	 	
 
		\begin{align*}
		 \nonumber \sup \;\; &-\alpha \epsilon + \frac{1}{N}\sum_{i \in [N]} s_i	\\ 
		 \nonumber  \st \;\;  &\alpha \in \mathbb R_+, s \in \mathbb R^N y^i, z^i \in \mathbb R^{n+3},  X^i, W^i \in \mathbb S^{n+3}_+ &&\forall i \in [N]\\
		 &y_0^i = -s_i+\alpha \hat \rho_i, \; y_1^i = s_i-\alpha-\alpha\hat\rho_i+R\mu-C, y_2^i=\alpha-R\mu\\
		 \nonumber& y^i_3, \cdots, y_n^i = 0, \ y^i_{n+1} = -s_i-\alpha\hat\rho_i-R\mu+C(n+1),\\
		 \nonumber&  y^i_{n+2} = -s_i+\alpha+\alpha\hat\rho_i+R\mu+C-C(n+1) , \ y^i_{n+3} = -\alpha && \forall i \in [N] \\ 
		 \nonumber & z_0^i = -s_i- \alpha\hat \rho_i, \; z_1^i = s_i+\alpha+\alpha\hat\rho_i+R\mu-C, z_2^i=-\alpha-R\mu\\
		 \nonumber& z^i_3, \cdots, z_n^i = 0, \ z^i_{n+1} = -s_i+\alpha\hat\rho_i-R\mu+C(n+1),\\
		 \nonumber&  z^i_{n+2} = -s_i-\alpha-\alpha\hat\rho_i+R\mu+C-C(n+1) , \ z^i_{n+3} = \alpha && \forall i \in [N] \\ 
		 \nonumber & \sum_{u + v = 2l - 1} x^i_{uv} = 0 && \forall l \in [n + 3] \ i \in [N]\\ \nonumber &
		\sum_{q = 0}^{l}\sum_{r=q}^{n+3+q-l} y^i_r {r \choose q} {n+3-r \choose l - q} a^{r-q}\hat\rho_i^q = \sum_{u + v = 2l} x^i_{uv} && \forall l \in [n+3] \cup \{0\} \ i \in [N] \\ \nonumber & \sum_{u + v = 2l - 1} w^i_{uv} = 0 && \forall l \in [n+3] \ i \in [N]\\ \nonumber &
		\sum_{q = 0}^{l}\sum_{r=q}^{n+3+q-l} z^i_r {r \choose q} {n+3-r \choose l - q} \hat\rho_i^{r-q}b^q = \sum_{u + v = 2l} w^i_{uv} && \forall l \in [n+3]\cup \{0\} \ i \in [N].
		\end{align*}		
	\end{thm}
	\begin{proof}
		
		The distributionally robust model with the ambiguity set \eqref{eq:wasser_ball} can be equivalently written as
    	\begin{align*}
    	\inf \quad  & \frac{1}{N} \sum_{i \in [N]} \int_\Xi f_n(\rho) \mathbb P_i(\rm d \rho) \\  \text{s.t.} \quad & \mathbb P_i \in \mathcal P_0(\Xi) \quad \forall i \in [N] \\ 
    	& \frac{1}{N} \sum_{i \in [N]} \int_{\Xi} \|  \rho -  {\hat \rho}_i \| \mathbb P_i(\rm d \rho) \leq \epsilon.
        \end{align*}	
		Its strong dual problem is given by~\cite[Theorem 4.2]{esfahani2017data}
		\begin{equation*}
		\begin{array}{ccll}
		&\displaystyle\sup _{\alpha \in \RR_+, \bm s \in \RR^N}&\displaystyle -\alpha \epsilon + \frac{1}{N} \sum_{i \in [N]} s_i\\
		&\st& \displaystyle s_i - \alpha \| \rho - \hat \rho_i\| \leq f_n(\rho)  \qquad \forall i \in [N] \ \forall \rho \in [a,b].
		\end{array} 
		\end{equation*}
		We can deal with each constraint separately for the cases $\rho \leq \hat \rho_i$ and $\rho \geq \hat \rho_i$, and consequently we have
		\begin{equation*}
		\begin{array}{ccll}
		&\displaystyle\sup _{\alpha \in \RR_+, \bm s \in \RR^N}&\displaystyle -\alpha \epsilon  + \frac{1}{N} \sum_{i \in [N]} s_i\\
		&\st& \displaystyle s_i + \alpha ( \rho - \hat \rho_i) \leq f_n(\rho)  \qquad \forall i \in [N] \ \forall \rho \in [a,\hat \rho_i]\\
		&&\displaystyle s_i - \alpha ( \rho - \hat \rho_i) \leq f_n(\rho)  \qquad \forall i \in [N] \ \forall \rho \in [\hat \rho_i,b].
		\end{array} 
		\end{equation*}
		Substituting the definition of $f_n(\rho)$ in \eqref{eq:social_rate} and applying algebraic reductions yield the following polynomial inequalities for each $i \in [N]$:
	    \begin{align}
		\label{so_const}
		\nonumber &(-s_i + \alpha \hat \rho_i) \rho^0 + (s_i - \alpha - \alpha \hat \rho_i +R\mu -C) \rho + (\alpha-R\mu) \rho^2 + (s_i - \alpha \hat \rho_i - R\mu +C(n+1)) \rho^{n+1} \\ \nonumber &\hspace{5.5em} + (-s_i+\alpha +\alpha \hat \rho_i +R\mu +C-C(n+1)) \rho^{n+2} -\alpha \rho^{n+3} \geq 0 \quad \forall \rho \in [a,\hat \rho_i], \\
		\nonumber &(-s_i - \alpha \hat \rho_i) \rho^0 + (s_i + \alpha + \alpha \hat \rho_i +R\mu -C) \rho + (-\alpha-R\mu) \rho^2 + (s_i + \alpha \hat \rho_i - R\mu +C(n+1)) \rho^{n+1} \\ \nonumber &\hspace{5.5em} + (-s_i-\alpha -\alpha \hat \rho_i +R\mu +C-C(n+1)) \rho^{n+2}+\alpha \rho^{n+3} \geq 0 \quad \forall \rho \in [\hat \rho_i,b]. \\
		\end{align}
		The inequalities are of the form $g_1^i(\rho) = \sum_{r = 0}^{n+3} y^i_r \rho^r \geq 0$ for $\rho \in [a,\hat \rho_i]$ and $g_2^i(\rho) = \sum_{r = 0}^{n+3} z^i_r \rho^r \geq 0$ for $\rho \in [\hat \rho_i,b]$, where $ y^i$ and $ z^i$ represent the coefficients of the respective polynomial inequalities. We next invoke the result of Lemma \ref{poly_lem_so1} for every $i \in [N]$ to express the inequalities in \eqref{so_const} as semidefinite constraints. This leads to  the desired semidefinite program, which completes the proof. 
	\end{proof}
		To determine an optimal joining threshold,  we compute the worst-case expected benefit rate ${Z}_s(n)$ for every $n \in \mathbb Z_+$. $1\leq n \leq n_e$,  using the result of Theorem \ref{wass_so}, and then select the  best threshold  $\hat n_s \in \argmax_{n\in \mathbb Z_+}\{ Z_s(n)\}$.
		
\subsection{Revenue Maximizer}
	The objective of a revenue maximizer is to find an optimal threshold $\hat n_r$ that maximizes the worst-case expected revenue rate of a firm, i.e., $\hat n_r \in \argmax_{n \in \mathbb Z_+} \{Z_r(n)\}$, where the worst-case expectation is computed over all the distributions in the Wasserstein ambiguity set $\mathcal B_\epsilon(\mathbb {\hat P} _N)$ defined by \eqref{eq:wasser_ball} with support set $\Xi = [a.b]$. The worst-case expected profit rate $ Z_r(n)$ is given by
	\begin{equation}
	\label{eq:dr_rm}
	Z_r(n) \coloneqq \inf_{\mathcal P \in \mathcal B_\epsilon(\mathbb {\hat P} _N)}  \mathbb E_{\mathbb P} \left[r_n(\rhot) \right].
	\end{equation}		
	
	\begin{thm}
	\label{wass_rm}
	For any $n \geq 1$, the worst-case expectation $Z_r(n)$ coincides with the optimal objective value of the following linear program:	
	
	\begin{equation*}
	\begin{array}{ccll}
	&\displaystyle \sup _{\alpha \in \RR_+, s \in \RR^N}&\displaystyle -\alpha \epsilon  + \frac{1}{N} \sum_{i \in [N]} s_i\\
	&\st& \displaystyle s_i + \alpha ( a - \hat \rho_i)  \leq r_n(a)  & \forall i \in [N]\\
	&& \displaystyle s_i  \leq r_n(\hat \rho_i)  & \forall i \in [N]\\
	&& \displaystyle s_i - \alpha ( b - \hat \rho_i) \leq r_n(b)  & \forall i \in [N].
	\end{array} 
	\end{equation*} 
	\end{thm}
	\begin{proof}
		The strong dual problem of $\inf_{\mathbb P \in \mathcal B_\epsilon(\mathbb {\hat P} _N)}  \mathbb E_{\mathbb P} \left[ r_n(\rhot) \right]$ is given by
		\begin{equation*}
		\begin{array}{ccll}
		&\displaystyle\sup _{\alpha \in \RR_+,  s \in \RR^N}&\displaystyle -\alpha \epsilon +  \frac{1}{N} \sum_{i \in [N]} s_i\\
		&\st& \displaystyle s_i - \alpha \| \rho - \hat \rho_i\|\leq r_n(\rho)  \qquad \forall i \in [N]\ \forall \rho \in [a,b].
		\end{array} 
		\end{equation*}
		Since the revenue rate function $r_n(\rho)$ is concave for $\rho \geq 0$, the semi-infinite constraints are satisfied if and only if each constraint is satisfied at three points $\rho=a,\hat \rho_i, b$, and consequently we have
		\begin{equation*}
		\begin{array}{ccll}
		Z_r(n) \coloneqq&\displaystyle\sup _{\alpha \in \RR_+,  s \in \RR^N}&\displaystyle -\alpha \epsilon + \frac{1}{N} \sum_{i \in [N]} s_i\\
		&\st& \displaystyle s_i + \alpha ( a - \hat \rho_i)\leq r_n(a)  \qquad& \forall i \in [N]\\
		&& \displaystyle s_i  + \alpha ( \hat \rho_i - \hat \rho_i)\leq r_n(\hat \rho_i)  \qquad& \forall i \in [N]\\
		&& \displaystyle s_i - \alpha ( b - \hat \rho_i) \leq r_n(b)  \qquad &\forall i \in [N].
		\end{array} 
		\end{equation*}
		Thus, the claim follows. 
	\end{proof}
	We compute the worst-case expected profit rate ${Z}_r(n)$ for every $n \in \mathbb Z_+$, $1 \leq n \leq \hat n_e$, using the result of Theorem \ref{wass_rm}, and obtain an optimal joining threshold $\hat n_r$ such that $\hat n_r \in \argmax_{n \in \mathbb Z_+}\{ Z_r(n)\}$.

\end{document}